\newif\ifpictures
\picturestrue

\documentclass[11pt]{amsart}
\usepackage{latex_base}
\usepackage{macros}
\newcommand{\NormalCone}[1]{\ensuremath{\operatorname{NC}{(#1)}}}
\newcommand{\NormalFan}[1]{\ensuremath{\operatorname{NF}{(#1)}}}
\newcommand{\Amoeba}[1]{\ensuremath{\mathcal{A}{(#1)}}}
\newcommand{\Minkowski}[3]{\ensuremath{\sum_{i = #2}^{#3}{#1_i}}}

\DeclareMathOperator{\Trop}{Trop}
\DeclareMathOperator{\Cayley}{Cay}

\DeclareMathOperator{\x}{\Vector{x}}
\DeclareMathOperator{\w}{\Vector{w}}

\author{Alperen A. Erg\"{u}r}

\address{Alperen A. Erg\"{u}r, University of Texas at San Antonio, One UTSA Circle, San Antonio, Texas, USA, 78249
\medskip}

\email{alperen.ergur@utsa.edu}

\author{Timo de Wolff}

\address{Timo de Wolff, Technische Universit\"at Braunschweig, Institut f\"ur Analysis und Algebra, AG Algebra, Universit\"atsplatz 2, 38106 Braunschweig,
 Germany\medskip}

\email{t.de-wolff@tu-braunschweig.de}



\title{A Polyhedral Homotopy Algorithm For Real Zeros}

\begin{document}

\begin{abstract}
We design a homotopy continuation algorithm, that is based on Viro's patchworking method, for finding real zeros of sparse polynomial systems. The algorithm is targeted  for polynomial systems with coefficients satisfying certain concavity conditions, it tracks optimal number of solution paths, and it operates entirely over the reals. In more technical terms; we design an algorithm that correctly counts and finds the real zeros of polynomial systems that are located in the unbounded components of the complement of the underlying A-discriminant amoeba. We provide a detailed exposition of connections between Viro's patchworking method, convex geometry of A-discriminant amoeba complements, and computational real algebraic geometry. 

\end{abstract}

\maketitle


\vspace{-0.4 in}

\section{Introduction}
\label{section:Introduction}
Let $\struc{\Vector{p}}=(p_1,p_2,\ldots,p_n)$ be a system of sparse polynomials in $\struc{\C[\Vector{x}]} = \C[x_1,\ldots,x_n]$ with support sets $A_1,A_2,\ldots,A_n \subseteq \Z^n$. More precisely, let
\begin{align*}
	 \struc{p_i} \ := \ \sum_{\alpha \in A_i} c_{\alpha}^{(i)} \Vector{x}^{\alpha} \; , \; \text{for} \; i=1,2,\ldots,n
\end{align*}
where $\Vector{x}^{\alpha}:=x_1^{\alpha_1} x_2^{\alpha_2} \ldots x_n^{\alpha_n}$. Bernstein's theorem from 1975 \cite{bernstein} shows that for generic choice of coefficients of $p_i$ the number of zeros of $\Vector{p}$ on $(\C^{*})^n$ equals to the \textit{\struc{mixed volume}} $\struc{\mathcal{M}(Q_1,Q_2,Q_3,\ldots,Q_n)}$ of the Newton polytopes $\struc{Q_i} := \conv(A_i)$. 

In the early 90's the \textit{\struc{polyhedral homotopy}} method was developed as an algorithmic counterpart of Bernstein's theorem \cite{huberpolyhedral}. The main idea of the polyhedral homotopy method is to continuously deform a given polynomial system to another ``easy'' system, that can be solved by pure combinatorics, and then trace back the change in the solution set with numerical path trackers. This geometric idea is colloquially referred to as \struc{\textit{toric deformation}}, and the ``easy'' systems with combinatorial structure are referred to as the systems at the \struc{\textit{toric limit}}.  Polyhedral homotopy method is currently implemented in several packages such as \textsc{PHCPack} \cite{PHCpack}, \textsc{Hom4ps-3} \cite{Hom4PS}, \textsc{pss5} \cite{Pss5},  and \textsc{HomotopyContinuation.jl} \cite{HomotopyJL}, and it has remarkable practical success.

For most applications of polynomial system solving, and for certain questions in theoretical computer science one needs to count and find zeros of polynomial equations over real numbers, e.g. see \cite{anne,koiran}. No general and efficient algorithm that counts real zeros of arbitrary sparse polynomial systems is known, and there are good complexity theoretic reasons to believe that at this level of generality the problem is intractable. Our aim is to locate a sufficiently general and tractable sub-case of real zero finding problem: Consider support sets $A_1,A_2,\ldots,A_n \subseteq \Z^n$ are given, can we find effectively checkable conditions on the coefficients of the equations that guarantee tractable solving over the reals? In other words: Where are the ``easy'' equations located in space of sparse real polynomial systems with $n$ equations and $n$ unknowns?

An important observation from real algebraic geometry suggests a map for ``easy'' polynomial systems: one can count real zeros by pure combinatorics if the polynomial system is at the ``toric limit''.  We informally state  this result (\textit{\struc{Viro's Patchworking Method for Complete Intersections}}) to motivate our discussion; see \cref{subsection:patchworking} for a precise statement. 
\begin{theorem} [Viro's Patchworking Method for Finitely Many Zeros] \label{viro-informal}
Let $A_1,\ldots,A_n \subseteq \Z^n$, let $\omega_i : A_i \rightarrow \R$ be lifting functions,
and consider the following family of equations parametrized by $t \geq 1$: 
\begin{align*} 
    p_i(t,\Vector{x}) \ := \ \sum_{\Vector{\alpha} \in A_i} c_{\Vector{\alpha}}^{(i)} t^{\omega_i(\Vector{\alpha})} \Vector{x}^{\Vector{\alpha}} \;  i=1,2,\ldots,n.
\end{align*}
Let $\varepsilon_i: A_i \rightarrow \{ -1, +1 \}$ be the sign functions defined by signs of the coefficients $c_{\Vector{\alpha}}^{(i)} \in \R$.  Then, for sufficiently large $t \gg 1$, the set of common zeros of $p_1(t,\Vector{x}),p_2(t,\Vector{x}),\ldots,p_n(t,\Vector{x})$ on $\R_{+}^n$ is homeomorphic to 
\begin{align*} \Trop(A_1,\omega_1,\varepsilon_1) \cap \Trop(A_2,\omega_2,\varepsilon_2) \cap \ldots \cap \Trop(A_n, \omega_n,\varepsilon_n) \end{align*}
\noindent where $\Trop(A_i,\omega_i,\varepsilon_i)$ are the positive part of tropical varieties $\Trop(A_i,\omega_i)$ as defined in \cref{subsection:patchworking} .
\end{theorem}
\cref{viro-informal} yields a polyhedral object that is homeomorphic to the common zero set of $p_1(t,\Vector{x}),\ldots,p_n(t,\Vector{x})$ on $\R_{+}^n$ for sufficiently large $t$, and it can also be used to handle the set of common zeros on $(\R^{*})^n$. We have three immediate questions:
\begin{enumerate}
 \item How can we quantify precisely when $t$ is ``sufficiently large''? 
 \item Given a polynomial system $p_1(t,\Vector{x}),\ldots,p_n(t,\Vector{x})$ with support sets $A_1,\ldots,A_n$ and coefficients $c_{\Vector{\alpha}}^{i}$ for $\Vector{\alpha} \in A_i$ (as in the theorem statement), can we guarantee that the number of common real zeros does not change as $t$ goes from $1$ to $\infty$?
 \item Can we use the technique in \cref{viro-informal} for polynomial systems that are not necessarily at the ``toric limit''?
\end{enumerate}

The first two questions are interrelated, and they form the main difficulty with respect to developing an algorithmic version of \cref{viro-informal}. These questions were asked since 90's \cite{berndsurvey}; to the best of our knowledge the current paper provides the first progress. We  provide an explicit criterion to answer the second question (stated in \cref{amoebamagic}).  The criterion also furnishes a homotopy algorithm that operates entirely over the reals, which we call \textit{\struc{real polyhedral homotopy algorithm (RPH)}}.  

The third question is due to  Itenberg and Roy; they conjectured Viro's patchworking method provides an upper bound for the number  of real zeros regardless of the polynomial system  being at the toric limit or not \cite{itenberg-roy}. Li and Wang provided a counterexample to the Itenberg-Roy Conjecture \cite{li-wang_descartes}.  

\subsection{Effective Patchworking}
Our development is based on an observation from the book \cite{gkz} by Gelfand, Kapranov, and Zelevinsky (henceforth GKZ) which provides a link between Viro's patchworking method and $A$-discriminants. 
Using the GKZ observation for an algorithm is not straightforward. It requires to locate a query point against the $A$-discriminant variety, and this may be intractable: The defining equation of the discriminant locus is known to be \textit{extremely} complicated; it obstructs the use of computational algebra methods. 
However, it is no obstruction against the use of amoeba theory. Discriminantal amoebas are proven to admit a certain parametric description, and it is easy to compute normal directions on their boundary; see \cref{subsection:Horn-Kapranov}. We exploit these special differential geometric properties of $A$-discriminant amoebas to develop an effective criterion for checking whether a given polynomial system is ``easy''. 
Note that RPH relies on notions from discrete and tropical geometry, and on further notions from GKZ. Furthermore, we use an algorithm called \textit{\struc{Tropical Homotopy}} due to Jensen \cite{jensentropical}. So, before reading the main statement in \cref{amoebamagic} we encourage the reader to check familiarity with the content of \cref{subsection:polyhedralcayley}, \cref{subsection:patchworking}, \cref{subsection:anders}, \cref{subsection:mixedcell}, and \cref{subsection:Horn-Kapranov}.  
\subsection{Complexity aspects}
Our work is inspired by the \textit{practical} efficiency of complex polyhedral homotopy algorithm. Complexity aspects of polyhedral homotopy have been elusive for more than two decades; early papers did not include any complexity analysis, later different authors approached the issue \cite{rojas-m,m1,m2}, certain technical obstacles still remain; see \cref{section:numericalcomplexity}.

A complete complexity analysis of RPH will only become possible when the scientific community fully understands the complexity of numerical path tracking for sparse polynomial systems. We present our thoughts on the complexity of discrete computations, and touch upon the complexity of the numerical part of RPH in  \cref{section:complexity}. 

We point out here that the main parameters  governing the complexity of RPH are different than its complex cousin: the overall complexity of RPH is controlled by the \textit{number of mixed cells} (a combinatorial quantity), the complexity of complex polyhedral homotopy is, in contrast, controlled by the \textit{mixed volume} (a geometric invariant). 
\subsection{Connections to fewnomial theory} 

A system of polynomials  $\Vector{p}=(p_1,p_2,\ldots,p_n)$ is called a \textit{\struc{patchworked polynomial system}} if the real zero set of $\Vector{p}$ is homeomorphic to a simplicial complex created by Viro's combinatorial patchworking technique. For instance, every polynomial system that passes our test in \cref{amoebamagic} is a patchworked system. In \cref{section:fewnomial} we observe the following result that is reminiscent to a conjecture from fewnomial theory \cite{khovanskiifewnomials} attributed to Kushnirenko \cite{kushnirenko}.

\begin{theorem} \label{few}
Let $\Vector{p}=(p_1,p_2,\ldots,p_n)$ be a  patchworked polynomial system where every polynomial $p_i$ has at most $t$ terms. Then $\Vector{p}$ can have at most $2^{n+1} \binom{n(t-1)}{n}$ many common zeros on $(\R^{*})^n$. 
\end{theorem}
Note that $2^{n+1} \binom{n(t-1)}{n} \leq 2^{n+1} e^n (t-1)^n$ where the right hand side resembles Kushnirenko's conjecture. To illustrate the difference between the number of paths tracked in RPH and the number of paths in complex polyhedral homotopy we provide a very simple example.
\begin{example}
Let $A=\{ (0,0,0) , (0,0,d) , (0,d,0), (d,0,0) \}$, and let $\Vector{p}=(p_1,p_2,p_3)$ where 
$p_i= a^{(i)}_0 + a^{(i)}_1 x_3^d + a^{(i)}_2 x_2^d + a^{(i)}_3 x_1^d$ are real polynomials with three variables. Further assume that the coefficients of $\Vector{p}$ is generic in the sense of Bernstein's theorem, this implies $\Vector{p}$ has $d^3$ many zeros on $(\C^{*})^3$. If $\Vector{p}$ is a patchworked polynomial system, by Theorem \ref{few} it has at most $1344$ zeros in $(\R^{*})^3$. We should note that $1344$ is very much an over estimation (correct bound in this specific example is at most $8$) where else $d^3$, with $d$ being arbitrarily large, is the exact number of zeros in $(\C^{*})^3$. 
\end{example}

\cref{few} is a direct application of McMullen's upper bound theorem. 
Things become geometrically more interesting when one tries to bound the number of mixed cells for support sets $A_i$ with different cardinalities (mixed supports). In \cite{bihandiscretemixed} it is claimed that a patchworked polynomial system $\Vector{p}=(p_1,p_2,\ldots,p_n)$, where $p_i$ has at most $t_i$ many terms, can have at most $\prod_{i=1}^n (t_i-1)$ many zeros on $\mathbb{R}_{+}^n$. We have learned from Bihan that the proof of this result is not correct, but the result still holds true. Bihan informed us that a new proof and an erratum will appear soon (\cite{Bihan:Corrigendum}).
\subsection{Structure of the paper}
Our aim is to write this paper as self contained as possible. The preliminaries section contains background information and results from discrete geometry, the theory of $A$-discriminants, symbolic computation, and numerical path trackers. Jensen's tropical homotopy algorithm and mixed cell cones are also introduced in this section. In the third section, we transform asymptotic and qualitative results from \cite{gkz} to a more quantitative and checkable condition. In the fourth section we present our real polyhedral homotopy algorithm and an example. The fifth section is concerned with the complexity aspects. The last section contains a discussion of questions, that were brought to our attention after the initial version of this paper appeared on ArXiv.  
\subsection*{Acknowledgments}
We cordially thank Sascha Timme for implementing a preliminary version of the algorithm developed in this article in the software \textsc{HomotopyContinuation.jl}, and for his help with developing \cref{Exa:OurAlgorithmInPractice}. We thank Mat\'ias Bender, Paul Breiding, Felipe Cucker, Mario Kummer, Gregorio Malajovich, Jeff Sommars, and Josue Tonelli-Cueto for useful discussions. The first author thanks J. Maurice Rojas for introducing him to the beautiful book \cite{gkz}. First author is supported by NSF CCF 2110075, and the second author is supported by the DFG grant WO 2206/1-1.

\section{Preliminaries}
\label{section:preliminaries}

We denote $\struc{[n]} := \{1,\ldots,n\}$, $\struc{\C^*} := \C \setminus \{0\}$, and $\struc{\R^*} := \R \setminus \{0\}$. Let $\struc{\Vector{e}_j}$ denote the $j$-th coordinate vector in $\R^{n}$. To avoid redundancies later in the articles we set $\struc{\Vector{e}_0} := \Vector{0}$.

For a given convex set $C$, we denote its boundary by $\struc{\partial C}$. For a convex cone $K \in \R^{n}$, the \struc{dual cone $K^{\circ}$} is defined as
\[ K^{\circ}:= \{ y \in \R^n : \langle x , y \rangle \geq 0 \; \text{for all} \; x \in K \}\]
For a given polytope $P$, we denote its \struc{\textit{vertex set}} as $\struc{\vertices{P}}$.
For $\Vector{v} \in \vertices{P}$  the \struc{\textit{normal cone}} is the collection of linear functional that achieves its maxima over $P$ at $\Vector{v}$ and is denoted with $\struc{\NormalCone{\Vector{v}}}$. Entire collection of all normal cones $\NormalCone{\Vector{v}}$ form a fan called  \struc{\textit{normal fan}} and denoted with $\struc{\NormalFan{P}}$. 

In what follows we consider finite sets $\struc{A} := \{\Vector{a}_1,\ldots,\Vector{a}_m\} \subset \Z^n$ and $\struc{A_1,A_2,\ldots,A_k} \subset \Z^n$, which are \textit{\struc{support sets}} of polynomials. We denote the \textit{\struc{Minkowski sum}} of the $A_i$ as $\struc{\Minkowski{A}{1}{k}}$. Note that
\begin{align*}
	\conv\left(\Minkowski{A}{1}{k}\right) \ = \ \sum_{i = 1}^k \conv(A_i).
\end{align*}

For a polynomial $p \in \C[\Vector{x}]$ with support $A$, the \textit{\struc{Newton polytope}} is given by $\struc{\New(p)} := \conv(A)$. We denote the \textit{\struc{variety}}, i.e. the common solutions of a system of polynomials $\p$ on complex numbers as $\struc{\variety{\p}}$, the \textit{\struc{real locus}} as $\struc{\varietyreal{\p}} := \variety{\p} \cap \R^n$, and positive / nonzero real locus as $\struc{\varietyrealpositive{\p}}$ and $\struc{\varietyrealnonzero{\p}}$.

\subsection{Polyhedral subdivisions, secondary polytope and Cayley configuration} \label{subsection:polyhedralcayley}
In this section we introduce polyhedral subdivisions, secondary polytopes and Cayley configurations; for further details we refer the reader to \cite{triangulations}.

Let $A \subset \Z^n$ be a set of lattice points and let $\struc{\omega}: A \rightarrow \R$ be a function. The \struc{\textit{lifting}} of $A$ induced by $\omega$ is defined as: 
\begin{align*} 
	\struc{A^{\omega}} \ := \ \left\{ (\x, \omega(\x)) : \x \in A  \right\}.
\end{align*} 
We call a face $F$ of $\conv(A^{\omega})$ an \struc{\textit{upper face}} if it is given by
\begin{align*}
	F \ = \ \{ \x \in \conv(A^{\omega}) : \langle \Vector{c} ,\x \rangle \geq \langle \Vector{c}, \Vector{y} \rangle \text{ for all } \Vector{y} \in \conv(A^{\omega}) \}
\end{align*}
where $\Vector{c}$ is a vector with a positive last entry. Intuitively, upper faces are the faces that are ``visible'' from $(0,\ldots,0,\infty)$. We project upper faces of $\conv(A^{\omega})$ on the point set $A$:  
\begin{align*} 
    \struc{\Delta_{\omega}} \ := \ \left\{ \x \in A : (\x,\omega(\x)) \text{ belongs to an upper face of } \conv(A^{\omega})  \right\}. 
\end{align*}
$\Delta_{\omega}$ is a polyhedral subdivision of $A$. Polyhedral subdivisions obtained this way are called \struc{\textit{coherent}} or \struc{\textit{regular}}. Note that $\Delta_{\omega}$ is a triangulation (i.e. a subdivision using only simplices) unless the lifted points $A^{\omega}$ have certain affine dependencies \cite[Remark 5.2.3]{triangulations}.

Now we define the \textit{secondary polytope} of $A$, which encodes all coherent triangulations of $A$, and discuss its key properties; see \cite[Section 5]{triangulations}.
\begin{definition}
	\label{definition:secondarypolytope}
	Let $T$ be a triangulation of $A =\{ \Vector{a}_1, \Vector{a}_2, \ldots, \Vector{a}_m \}$, and let $\sigma_1,\ldots,\sigma_s$ be the simplices in $T$. We define
    \begin{align*} 
        \struc{\Phi_A(T)} \ := \ \sum_{j=1}^m \left(\sum_{\{\sigma \in T \, : \, \Vector{a}_j \in \sigma\}} \vol(\sigma) \right) \Vector{e}_j.
    \end{align*} 
    We define the \struc{\textit{secondary polytope}} of $A$ as:
    \begin{align*} 
        \struc{\secondpoly{A}} \ := \ \conv \left\{ \Phi_A(T) \ : \ T \text{ is a triangulation of } A \right\}.
    \end{align*}
    The corresponding normal fan $\NormalFan{\secondpoly{A}}$ is called the \textit{\struc{secondary fan}}. For its cones, the \textit{\struc{secondary cones}}, we use the abbreviated notation $\struc{\NormalCone{T}} := \NormalCone{\Phi_A(T)}$.
\end{definition}

\begin{theorem} \cite[Section 5]{triangulations}. \label{theorem:secondarypolytope}
The secondary polytope has the following properties:
\begin{enumerate}
    \item The vertices of $\secondpoly{A}$ are in one to one correspondence to the coherent triangulations of $A$. 
    \item The face lattice of $\secondpoly{A}$ is isomorphic to a refinement poset of the coherent polyhedral subdivisions of $A$.
    \item A lifting function $\omega: A \rightarrow \R$ induces the triangulation $T$ if and only if $\omega \in \interior{\NormalCone{T}}$.
    \item Consider the support set $A$ as a $n \times m$ integer matrix. Then every secondary cone $\NormalCone{T}$ includes the $n+1$ dimensional linear space spanned by rows of $A$ and all ones vector $(1,1,\ldots,1)$.  As a consequence, the secondary polytope $\secondpoly{A}$ is $m-n-1$ dimensional.
\end{enumerate}
\end{theorem}

For later use, we need to have a better understanding of the description of secondary cones $\NormalCone{T}$. We first define a circuit.
\begin{definition} \label{circuit}
An affine dependence among lattice points of a set $A \subseteq \mathbb{Z}^n$ is the relation given by  $\sum_{\alpha \in A} a_{\alpha} \alpha =0$ where $\sum_{\alpha \in A} a_{\alpha}=0$. A circuit $Z$ is a collection of affinely dependent lattice points where every proper subset of $Z$ is affinely independent. Consequenlty, circuit represents a unique (up to scaling) affine relation $\sum_{\alpha \in Z} \lambda_{\alpha} \alpha =0$ and $\sum_{\alpha \in Z} \lambda_{\alpha} = 0$
\end{definition}
A rhombus in the plane is a nice example of a circuit. The following is a basic fact about circuits, see e.g. Lemma 2.4.2 \cite{triangulations}.
\begin{lemma}
Let $Z$ be a circuit, then $Z$ can be decomposed into a disjoint union of two sets $Z = Z_{+} \cup Z_{-}$ with the following property: 
\[ \mathcal{Z}_{+} = \{ Z - \{ \alpha \} : \alpha \in Z_{+} \} \; \; , \; \; \mathcal{Z}_{-} = \{ Z - \{ \alpha \} : \alpha \in Z_{-} \}\]
are the two triangulations of $Z$. 
\end{lemma}
The volume of the simplex  $Z-\{ \alpha \} $ is equal to absolute value of a determinant. Let us denote this determinant with $\sigma_{\alpha}$. A standard argument, see e.g. Remark 4.1.8 in \cite{triangulations}, shows that $\sigma_{\alpha}$'s determine the unique affine relation supported by $Z$. That is, using the terminology in \cref{circuit}, we have $\sigma_{\alpha}=\lambda_{\alpha}$ (up to swapping $Z_{-}$ and $Z_{+}$). 

Now consider a regular triangulation $T$ and suppose $\omega$ is a lifting function that induces $T$. Then, for a simplex $conv \{ a_1,a_2,\ldots,a_{n+1} \}$ in $T$, and $a_{n+2} \in A$ with $a_{n+2 } \notin conv \{ a_1,a_2,\ldots,a_{n+1} \}$, we must have that $(a_{n+2},\omega(a_{n+2}))$ lies `above' the affine spane of $(a_1,\omega(a_1)), \ldots , (a_{n+1}, \omega(a_{n+1})) )$. Suppose $\sum_{i=1}^{n+2} \lambda_i a_i = 0$ is the unique affine relation of the circuit $\{ a_1,a_2,\ldots,a_{n+2} \}$, if we have $\sum_{i=1}^{n+2} \lambda_i \omega(a_i) =0$ then we know that $(a_{n+2},\omega(a_{n+2}))$ is in the affine span of the vectors $\{ (a_1, \omega(a_1)), \ldots,(a_{n+1}, \omega(a_{n+1})\}$. To have $(a_{n+2},\omega(a_{n+2}))$ `above' simply corresponds to $\sum \lambda_i \omega(a_i) >0$. In conclusion, the secondary cone $\NormalCone{T}$ is described by inequalities supported on circuits, and these inequalities are of the form $\sum \lambda_i \omega(a_i) >0$ where $\lambda_{\alpha}$ are signed volumes of simplices in the triangulation of the circuit (up to scaling).

Now we consider polyhedral subdivision of a set $A$ where  $A = \Minkowski{A}{1}{k}$. Let $\struc{F}$ be a cell in coherent polyhedral subdivision of $\Minkowski{A}{1}{k}$ introduced by a lifting function $\omega$.  Then $F$ corresponds to  a face in $\sum_{i = 1}^n \conv(A_i)^{\omega}$. Let $F = \Minkowski{F}{i}{k}$ where $\struc{F_i}$ are the corresponding faces on $\conv(A_i)^{\omega}$. 

\begin{definition}
\label{definition:finemixed}
	A coherent polyhedral subdivision $\struc{\Delta_{\omega}}$ of $A_1+A_2+\ldots+A_k$ for $A_i \subset \Z^n$ is called \textit{\struc{fine mixed}} if it satisfies the following conditions:
    \begin{enumerate}
    \item For all cells $F$ in the subdivision, we have $\sum_{i=1}^k \dim(F_i)=n$, and  
    \item for all cells $F$ in the subdivision we have $\sum_{i=1}^k (\# F_i - 1)=n$,
    \end{enumerate}
    where $\struc{\# F_i}$ denotes the number of vertices of $F_i$.
\end{definition}

We also need to define Cayley configuration of point sets $A_1,A_2,\ldots,A_k$ and the corresponding Cayley polytope.
\begin{definition}
    \label{Definition:Cayley}
	We define the \struc{\textit{Cayley configuration}} of   $A_1,A_2,\ldots,A_k$ as
	\begin{align*} 
        \struc{\A} \ = \ \struc{A_1 * A_2 * \cdots * A_k} \ := \ \{ (\Vector{x},\Vector{e}_{i-1}) : \Vector{x} \in A_i \} \subseteq \R^{n+k-1},
    \end{align*}
    The \textit{\struc{Cayley polytope}} is defined as $\conv(\A)$, denoted by $\struc{\Cayley(\A)}$.
  \end{definition}
    
The following observation is implicit in most papers in literature: A natural slicing of the Cayley polytope $\Cayley(\A)$ is equivalent to $\sum_{i = 1}^n \conv(A_i)$. More precisely, consider the following set defined by the intersection of $\Cayley(\A)$ with several hyperplanes:    
    \begin{align*} 
     \struc{\tilde{\Cayley(\A)}} \ := \ \left\{ \Vector{x} \in \Cayley(\A) \ : \ x_{n+1}=x_{n+2}=\ldots=x_{n+k-1}=\frac{1}{k} \right\}.
    \end{align*}
    
Observe that a $k$-scaling of $\tilde{\Cayley(\A)}$, i.e., $k \cdot \tilde{\Cayley(\A)}$, is equal to $\sum_{i = 1}^n \conv(A_i)$. For a detailed explanation and a picture-proof see  \cite{hubercayley}.

Suppose that $T$ is a coherent triangulation of the Cayley configuration $\A$.  First, note that $T \cap \tilde{\Cayley(\A)}$ creates a polyhedral subdivision of $\tilde{\Cayley(\A)}$. Via the equivalence, this gives a polyhedral subdivision of $\sum_{i = 1}^n \conv(A_i)$. Let $\sigma$ be a simplex in $T$, then $\sigma$ has $n+k$ vertices which split into sets of vertices $\sigma_i$ that are induced by $A_i$. None of the $\sigma_i$ are empty since otherwise $\sigma$ can not be full-dimensional. Then, up to an isomorphism, $F_{\sigma}= \conv(\sigma_1) + \conv(\sigma_2) + \ldots + \conv(\sigma_k)$ yields a cell in the polyhedral subdivision of $\sum_{i = 1}^n \conv(A_i)$, and all such cells yield a fine mixed subdivision of $\sum_{i = 1}^n \conv(A_i)$.  This correspondence gives a bijection between coherent triangulations of the Cayley polytope and coherent fine mixed subdivisions of the Minkowski sum $\sum_{i = 1}^k \conv(A_i)$; see \cite[Theorem 5.1]{berndresultant}. 

In summary, coherent fine mixed subdivisions of $\sum_{i = 1}^k \conv(A_i)$ are encoded by the vertices of the secondary polytope $\Sigma(\A)$ and the corresponding secondary cones. 
\begin{remark}
 At various parts of this article (in our theorem statements and algorithms) we work with triangulations. For a  generic lifting function $\omega$ the induced polyhedral subdivision $\Delta_{\omega}$ is a triangulation.  The proof of \cite[Proposition 2.2.4]{triangulations} suggests an algorithm, albeit an inefficient one, to check whether a given lifting is generic. The question of finding an efficient algorithm to check genericity of a lifting is an interesting question, but it lies beyond the scope of our paper.
\end{remark}
\subsection{Viro's patchworking method} 
\label{subsection:patchworking}
In this section we introduce Viro's patchworking method for complete intersections.
For further details and relations to Hilbert's 16th problem, we kindly refer the reader to Viro's survey \cite{virosurvey}. For further background information on tropical geometry see e.g., \cite{rau,mikhalkin,Maclagan:Sturmfels}. For implementations of patchworking technique please see \cite{ViroSage} and \cite{joswigviro}.
\begin{definition}
    \label{definition:tropicalvariety}
    Let $A=\{ \Vector{a}_1, \Vector{a}_2, \ldots, \Vector{a}_m \} \subset \Z^{n}$ and $\Delta_{\omega}$ be a coherent triangulation of $A$ given by a lifting function $\omega : A \rightarrow \R$. We define the associated \textit{\struc{tropical variety}} as
    \begin{align*} 
        \struc{\Trop(A,\omega)} := \{ \Vector{x} \in \R^n :  \max_{i} \{ \langle \Vector{x} , a_i \rangle + \omega(a_i) \}  \; \text{is attained at least twice} \}.
    \end{align*}
\end{definition}
Since we are interested in real varieties, we need to distinguish a positive and a negative part of $\Trop(A,\omega)$. Observe that $\Trop(A,\omega)$ together with its complement creates a polyhedral decomposition of $\R^n$.  Also, by definition, every full-dimensional cell in the complement of $\Trop(A,\omega)$ corresponds to a unique $\Vector{a}_j \in A$ as it is given by the set:
\begin{align*}
	\left\{\Vector{x} \in \R^n \ : \ \langle \Vector{x} , \Vector{a}_j \rangle + \omega(a_j) > \langle \Vector{x} , \Vector{a}_i \rangle + \omega(a_i)  \text{ for all } i \in [n] \setminus \{j\}\right\}.
\end{align*}
We define the sign of this cell as $\struc{\varepsilon(\Vector{a}_j)}$.  For every $(n-1)$-dimensional cell in $\Trop(A,\omega)$, there exist two adjacent $n$-dimensional cells with signs assigned by $\varepsilon$.  This motivates the definition of the positive part of a tropical variety. 
\begin{definition}
\label{definition:positiveparttropicalvariety}
The \textit{\struc{positive part}} $\struc{\Trop(A,\omega,\varepsilon)}$ of a given tropical variety $\Trop(A,\omega)$ is the subcomplex consisting of those $(n-1)$-dimensional cells that are adjacent to two $n$-cells with different signs. 
\end{definition}

\begin{theorem} [Viro's Patchworking for Complete Intersections \cite{berndviro}] 
\label{Theorem:BerndViro}
Let $A_1,\ldots,A_k \subset \Z^n$, let  $\omega: A_1 * A_2 * \ldots * A_k \rightarrow \R$ be a lifting function.
Consider a system of polynomials $\struc{\Vector{p}}=(p_1,p_2,\ldots,p_k)$ defined as follows:
\begin{align*} 
    \struc{p_i(t,\x)} \ := \ \sum_{\alpha \in A_i} c_{\alpha} t^{\omega(\alpha)} \x^{\alpha} 
\end{align*}
with $c_{\alpha} \in \R$. Let $\struc{\varepsilon}: A_1 * A_2 * \cdots * A_k \rightarrow \{ -1, +1 \}$ be the sign function defined by coefficients of $\Vector{p}$. 
Then, for sufficiently large $t>>1$, the real algebraic set $\varietyrealpositive{\p}$ is homeomorphic to 
\begin{align*} 
	\Trop(A_1,\omega_1,\varepsilon_1) \cap \Trop(A_2,\omega_2,\varepsilon_2) \cap \ldots \cap \Trop(A_k,\omega_k,\varepsilon_k) 
\end{align*}
where $\omega_i$ and $\varepsilon_i$ are restrictions of $\omega$ and $\varepsilon$ to $A_i$.
\end{theorem}

\begin{remark}
 For readers who are familiar with non-Archimedian tropical geometry the theorem statement here might look confusing. The only difference is that in non-Archimedian tropical geometry it is customary to use min notation, lower facets, and $t$ tends to zero. In amoeba theory, however, it is customary to use max notation, upper facets, and $t$ tends to $\infty$. We follow the amoeba theory convention.
\end{remark}

\cref{Theorem:BerndViro} generalizes to the set of zeros on the appropriate toric variety by applying the theorem on every one of the $2^n$ orthants separately and then gluing them together; see \cite[Theorem 5]{berndviro}. We illustrate \cref{Theorem:BerndViro} on the most simple example possible.

\begin{example} 
The set $A :=\{ \Vector{e}_0,\Vector{e}_1,\ldots, \Vector{e}_n \}$ represents the support set for linear forms (and hence its convex hull is the standard simplex).  We consider positive solutions of an affine linear form $f = u_0+\sum_{i=1}^n u_i x_i$, i.e., the solutions with $x_i > 0$. 
We use a variant of moment the from symplectic geometry called \textit{\struc{algebraic moment map}}:
\begin{align*} 
	\struc{\mu_A}: \R_{+}^n \rightarrow \conv(A) \qquad \x \ \mapsto \ \frac{ \sum_{i} x_i \Vector{e}_i }{1+ \sum_{i} x_i }.
\end{align*}
This map is a homeomorphism. The image of  $\varietyrealpositive{u_0+\sum_{i=1}^n u_i x_i}$ under $\mu_A$ is given by:
\begin{align*} 
	\mu_A(\varietyrealpositive{f}) \ = \ \left\{ (y_1,y_2,\ldots,y_n) \in \conv(A) :  u_0\left(1-\sum_{i=1}^n y_i\right) + \sum_{i=1}^n u_i y_i = 0 \right\}. 
\end{align*}
Hence, $\mu_A(\varietyrealpositive{f})$ is defined by the linear form $u_0+u_1x_1+\ldots+u_nx_n$ on the simplex $\conv(A)$, and it separates those $\Vector{e}_i$ with $u_i > 0$ from those $\Vector{e}_j$ with $u_j < 0$.
\end{example}

To prove \cref{Theorem:BerndViro} above, one replaces the simplex with the triangulation, and the moment map with the moment map corresponding to the toric variety defined by $A_1+A_2+\ldots+A_k$ as explained in of \cite[Chapter 11, Section 5, Subsections C and D]{gkz}.
We provide another example, first considered by Sturmfels \cite[Page 382]{berndresultant}. 

\begin{example}
	Consider the two polynomials
	\begin{align*}
		& f_t \ = \ x_2^3 - tx_1x_2^2 - t^5x_1^2x_2 + t^{12}x_1^3 - tx_2^2 + t^4 x_1 x_2 - t^9 x_1^2  - t^5 x_2 - t^9 x_1 + t^{12}  \\
		& g_t \ = \ t^8 x_2^2 - t^6 x_1 x_2 + t^6 x_1^2 - t^3 x_2 - t^2 x_1 + 1
	\end{align*}
We consider the lifting function $\omega$ introduced by the exponents of $t$, the sign function $\varepsilon$ introduced by the coefficients of $f_t$ and $g_t$, and compute the corresponding patchworking. We present the outcome in \cref{Figure:ViroPatchworking}. The computation was already carried out by Sturmfels in the original article \cite{berndresultant} in '94. Here, we generate a plot using the \textsc{Viro.sage} package by O'Neill, Kwaakwah, and the second author \cite{ViroSage}.\label{Exa:ViroPatchworkingCompleteIntersection}
\end{example}
\begin{figure}[ht]
	\ifpictures
	\begin{center}
		\includegraphics[width=0.45\linewidth]{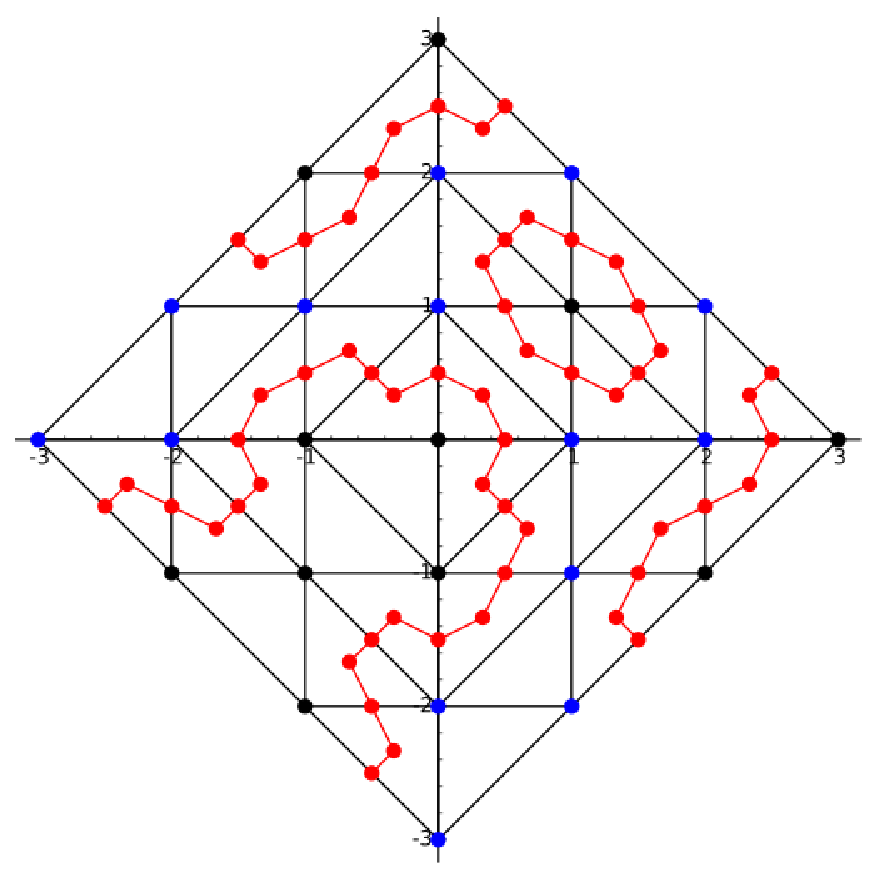} \qquad
		\includegraphics[width=0.45\linewidth]{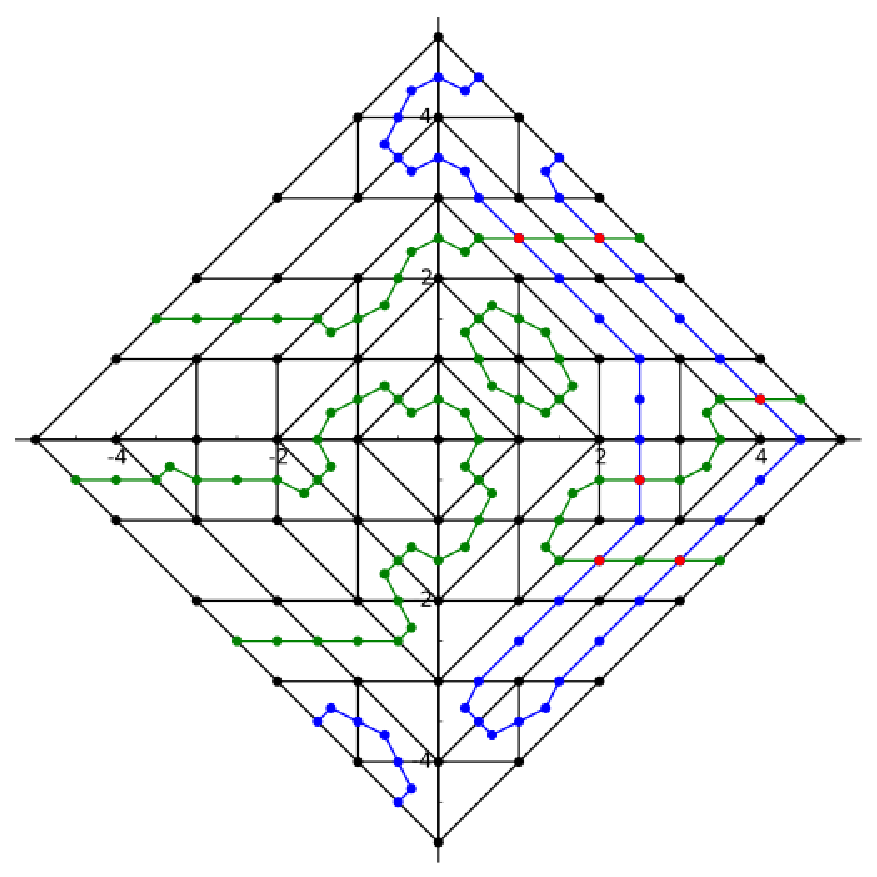}
	\end{center}
    \caption{Viro Patchworking of $f_t$ and the complete intersection of $f_t$ and $g_t$ for $f_t,g_t$ defined as in \cref{Exa:ViroPatchworkingCompleteIntersection}.}
  \label{Figure:ViroPatchworking}
\end{figure}

\subsection{Mixed-cell cones and Jensen's tropical homotopy algorithm}
\label{subsection:anders}
In this article we are concerned with zero dimensional systems, that is we have $n$ support sets $A_1,A_2,\ldots,A_n \subset \Z^n$. For this case, some cells in the regular triangulation $\Delta_\omega$ (introduced by a lifting $\omega$) of  $\A=A_1 * A_2 * \cdots * A_n$  are of particular interest. These cells are called mixed cells. Formally, a cell $\sigma \in \Delta_\omega$ that has $2$ elements from each $A_i$ is called a \struc{\textit{mixed cell}}.  Equivalently, after identification as explained \cref{subsection:polyhedralcayley}, in a fine mixed subdivision of $A_1+A_2+\ldots+ A_n$, mixed cells $\sigma$ are the cells that are given by the Minkowski sum of $n$ edges. 

On the dual side, when we consider the finite set of points in the intersection  
\begin{align*} \Trop(A_1,\omega_1,\varepsilon_1) \cap \Trop(A_2,\omega_2,\varepsilon_2) \cap \ldots \cap \Trop(A_n,\omega_n,\varepsilon_n) \end{align*} 
each of these points correspond to a mixed cell in the triangulation of $\A=A_1 * A_2 * \cdots * A_n$ where every two vertices from each $A_i$ have opposite signs. In the current literature, such a simplex is called an \textit{\struc{alternating mixed cell}} \cite{jensentropical}.

The main observation here is that \cref{Theorem:BerndViro}  depend only on the mixed cells in a triangulation of $\A=A_1*A_2*\ldots*A_n$; we do not need to differentiate between two triangulations that have same collection of mixed cells. We formalize this as follows. 
\begin{definition}[Mixed-Cell Cone of a Triangulation]
Let $T$ be a triangulation of $\A=A_1 * A_2 * \ldots *A_n$, and let $\sigma \in T$ be a mixed cell. For every lifting function $\omega: \A \rightarrow \R$ (represented by a vector in $\R^{\A}$) we denote the induced subdivision as $\Delta_{\omega}$. We define the \struc{\textit{mixed cell cone of} $\sigma$} as:
\begin{align*} 
    \struc{M(\sigma)} \ := \ \{ \omega \in \R^{\A} : \sigma \text{ is a mixed-cell in } \Delta_{\omega} \}.
\end{align*}
Moreover, we define the \struc{\textit{mixed cell cone of} $T$} as:
\begin{align*} 
    \struc{M(T)} \ := \ \bigcap_{\{\sigma \ : \ \sigma \text{ is mixed cell of } T\}} M(\sigma).
\end{align*}
\end{definition}

To clarify the difference between the mixed cell cone and the secondary cone we make a definition and state a lemma.
\begin{definition}
    Let $A_1,A_2,\ldots,A_n \subset \mathbb{Z}^n$ and set $\A = A_1 * A_2 * \ldots * A_n \subset \mathbb{Z}^{2n-1}$. If $\Gamma$ is a facet of $\A$ defined by $\Gamma = \{ x \in \A : x_{n+i} = 0 \}$ for some $1 \leq i \leq n-1$, or $\Gamma = \{ x \in \A : \exists \; i \; \text{such that} \; 1 \leq i \leq n-1  \; \text{and} \; x_{n+i}= 1 \}$ then we call $\Gamma$ an irrelevant facet. If $\Gamma$ is a face included in an irrelevant facet, we call $\Gamma$ an \struc{irrelevant face}. 
\end{definition}

\begin{lemma} 
\label{Lemma:Mixed-cell-cone}
    Let $T=\Delta_{\omega}$ for a lifting function $\omega$, and assume that $\omega \in \NormalCone{\Vector{v}}$ where  $\Vector{v}$ is a vertex of the Newton polytope of the $\A$-discriminant and $\NormalCone{\Vector{v}}$ is its normal cone. Then, we have
    \[ \NormalCone{\Vector{v}}^{\circ}  \subseteq M(T)^{\circ} \subseteq \NormalCone{T}^{\circ}  \]
    Moreover, if $\tau \in \NormalCone{T}^{\circ} -  M(T)^{\circ} $ then $\tau$ is supported on a set that is included in the union of irrelevant faces of $\A$. 
\end{lemma}

\begin{proof}
Inclusion of the cones follow directly from definition, for further structural information we refer to $D$-equivalence notion in Chapter 11, Section 3, subection B of \cite{gkz}.

We prove ``moreover'' part of the claim. Let  $\tau \in \NormalCone{T}^{\circ} -  M(T)^{\circ} $ be an inequality supported on a circuit $Z$. We claim in $Z$ there exist an $i \in [n]$ such that  $\abs{Z \cap A_i}=1$.  Assume otherwise, then we have that for some $j$: $\abs{Z \cap A_i}=2$ for all $i \neq j$, and $\abs{Z \cap A_j}=3$. Then, passing from one triangulation of $Z$ to another involves a mixed cell change which contradicts with the assumption $\tau \in \NormalCone{T}^{\circ} -  M(T)^{\circ} $. Now without loss of generality assume $Z \cap A_1=\Vector{\alpha}$. Then $Z - \Vector{\alpha}$ lies in an irrelevant face of $\A$, and the lattice distance from $\Vector{\alpha}$ to affine hull of $Z - \Vector{\alpha}$ is $1$. One can easily observe that the simplex $\sigma_{\alpha}$ corresponding to $\alpha$ in $Z$ is not full-dimensional and hence has volume zero. Thus, the inequality $\tau$ is supported on $Z-\alpha$. In general, any element of $\NormalCone{T}^{\circ} -  M(T)^{\circ} $ is a conic combination of circuit inequalities $\tau \in \NormalCone{T}^{\circ} -  M(T)^{\circ}$ and we showed that such $\tau$ are supported on irrelevant faces.
\end{proof}
In the rest of this paper we will use the mixed-cell cone $M(T)$ and it will be represented by circuit inequalities generating the cone, that is we work with $M(T)^{\circ}$. Luckily for us, there is already an efficient algorithm for computing $M(T)^{\circ}$: Jensen's \struc{\textit{tropical homotopy algorithm}}, see \cite{jensentropical}, computes for a given (generic) lifting function $\omega$, and point configurations $A_1,A_2,\ldots,A_n$, the triangulation $T=\Delta_{\omega}$ of $\A=A_1*A_2*\ldots*A_n$ and $M(T)^{\circ}$. The idea of Jensen's algorithm is to start from a lifting function $\beta$ yielding only one mixed cell. Then, one keeps track of the changes in the mixed-cell cone  as one changes the lifting function linearly from $\beta$ to a target lifting $\omega$. The algorithm updates the mixed-cell cone with the violated circuit inequalities, and halts whenever it arrives at a triangulation $T$ with $\omega \in M(T)$. The correctness of the algorithm follows from the fact that changes in the regular triangulations always happen by a change between two triangulations of a circuit, and every such change corresponds to one circuit inequality added to the mixed-cell cone.
\subsection{Solving binomial systems over the reals} 
\label{subsection:binomials}
Since we repeat the Viro construction in every orthant of $(\R^{*})^n$, the sign vector $\varepsilon$ changes. However, the lifting function $\omega$ and the corresponding triangulation remains the same for all orthants.  So, in order to count the number of real zeros with Viro's method, one needs to investigate the mixed cells and check how many times a mixed cell becomes an alternating one. Algorithmically, instead of going through Viro's construction $2^n$ times, it is more convenient to use  \textit{\struc{binomial systems}}, i.e., systems of polynomials where every polynomial has only two terms. Every mixed cell corresponds to a binomial system,  and solving that binomial system on $(\R^{*})^n$ corresponds to counting how many times the mixed cell becomes and alternating mixed cell. This approach is much more effective. 

Now we outline how to solve binomial systems over the reals. Consider the following system of binomials:
\begin{align*}
c_{11}\x^{\Vector{a}_{11}} \ = \ c_{12} \x^{\Vector{a}_{12}}, \quad c_{21}\x^{\Vector{a}_{21}} \ = \ c_{22} \x^{a_{22}}, \ \ldots, \quad c_{n1}\x^{\Vector{a}_{n1}} \ = \ c_{n2} \x^{\Vector{a}_{n2}}     
\end{align*}
where $c_{ij} \in \R^{*}$ and $\Vector{a}_{ij} \in \Z^n$. This system is equivalent to the following system of equations:
\begin{align} \label{Equation:2}
 \x^{\Vector{a}_{11}-a_{12}} \ = \ \frac{c_{12}}{c_{11}}, \quad \x^{\Vector{a}_{21}-a_{22}} \ = \ \frac{c_{22}}{c_{21}}, \ \ldots, \quad \x^{\Vector{a}_{n1}-a_{n2}} \ = \ \frac{c_{n2}}{c_{n1}}
\end{align}
Set $\Vector{d}_i=\Vector{a}_{i1}-\Vector{a}_{i2}$, and $D=[\Vector{d}_1 \Vector{d}_2 \ldots \Vector{d}_n]$. To solve the system \cref{Equation:2} over $(\R^{*})^n$, it suffices to perform the elementary integer operations  that reduce $D$ into its Hermite normal form. This operations can be done in strong polynomial time \cite{kannan}. The result is a system of equations 
in the following format:
\begin{align} \label{Equation:3}
	\x_1^{\Vector{h}_{11}} \ = \ \lambda_1 , \quad \x_1^{\Vector{h}_{21}} \x_2^{\Vector{h}_{22}} \ = \ \lambda_2 , \ \ldots , \quad \x_1^{\Vector{h}_{n1}} \ldots \x_{n}^{\Vector{h}_{nn}} \ = \ \lambda_n,
\end{align}
where $\Vector{h}_{ij} \in \mathbb{Z}$ and $\lambda_i \in \R^{*}$. The solutions of \cref{Equation:3} are completely determined by the signs of $\lambda_i$ and $\Vector{h}_{ij}$ being even or odd. 
Hence, \cref{Equation:3} either has no solution in $(\R^{*})^n$, or there exist solutions differing only by their signs.

There is also a recent paper focusing on probabilistic analysis of numerical methods for binomial system solving \cite{cmonmaurice}.

\subsection{$A$-discriminants} \label{gkzbasics}
Given a set of lattice points $A=\{ \Vector{a}_1, \Vector{a}_2, \ldots, \Vector{a}_m \} \subset \Z^n$, we define 
\begin{align*}
	\struc{\C^A} \ := \ \left\{\sum_{\Vector{\alpha} \in A} c_{\Vector{\alpha}} \Vector{x}^{\Vector{\alpha}} \in \C[\Vector{x}] \ : \ c_{\Vector{\alpha}} \in \C \text{ for all } \Vector{\alpha} \in A\right\}
\end{align*}
as the \struc{\textit{space of polynomials supported on} $A$}. Note that $\C^A$ is isomorphic to $\C^{m}$ with $m = \# A$. We define $\struc{(\C^*)^A}$ analogously with $c_{\Vector{\alpha}} \in \C^*$. Now we define our protagonist the $A$-discriminant variety:
\begin{align*} 
    \struc{\nabla_A} \ := \ \ovl{ \left\{ f \in (\C^*)^A \  : \ f \text{ has a singularity on } (\C^{*})^n \right\} }.
\end{align*}
$\struc{\nabla_A}$ is a cone over a projective variety, to get a better sense of this we need to introduce projective toric variety corresponding to $A$:  
\[ X_A := \ovl{ \{ [x^{a_1}:x^{a_2}:\ldots:x^{a_m}] : x \in (\C^{*})^n \} } \]
One can observe that $X_A$ is essentially Zariski closure of a torus orbit. A polynomial $f$ supported on $A$ can be considered as a linear form on the toric variety $X_A$. Discriminant variety corresponds to the hyperplanes that intersect the toric variety $X_A$ non-transversally. This means if $f \notin \struc{\nabla_A}$ then the zero set of $f$ has no singularity on $X_A$. This also means $ \struc{\nabla_A}$ is the cone over the projective dual of $X_A$. Thus, except for specific degenerate configurations $A$, $\struc{\nabla_A}$ is an irreducible hypersurface given by a polynomial with integral coefficients; \cite[Chapter 9]{gkz}. We denote the defining equation of $\struc{\nabla_A}$ with $\Delta_A$. We are interested in the real part of the discriminant variety
\begin{align*}
	\struc{\nabla_A(\R)} \ := \ \nabla_A \cap \R[\Vector{x}]
\end{align*}
The hypersurface $\nabla_A(\R)$ partitions the coefficient space $\R^{A}$ into connected components. If two polynomials $f,g \in \R^{A}$ lie in the same connected component of $\R^{A} - \nabla_A(\R)$ and $X_A$ is smooth then the zero sets of $f$ and $g$ on are isotopic on $X_A$ \cite[pg 380]{gkz}.

For the purposes of homotopy continuation we want to have zero sets of $f,g \in \R^{A}$  be isotopic in the torus orbit $X_A^{\circ}$ instead of the compactification $X_A$.  The nice fact is that $X_A$ admits a decomposition into a disjoint union of torus orbits:
\[ X_A = \sqcup_{\Gamma} X_{\Gamma}^{\circ} \]
where $\Gamma$ are faces of the polytope $conv(A)$ and $X_{\Gamma}^{\circ}$ denotes the torus orbit for which the toric variety $X_{\Gamma}$ is the closure. So to have $\varietyrealnonzero{f}$ and $\varietyrealnonzero{g}$ isotopic we will require two conditions: $f$ and $g$ both does not have a zero on $X_{\Gamma}^{\circ}$ for any proper face $\Gamma$ of $conv(A)$, and $f$ and $g$ are isotopic on $X_A$. Since singularities of toric variety $X_A$ are known to be on $X_{\Gamma}^{\circ}$ for co-dimension two or higher faces, once the first condition is guaranteed the second condition boils down to $f$ and $g$ being in the same connected component in $\R^{A} - \nabla_A(\R)$.

At this point, we need to introduce  \struc{\textit{sparse resultant}}. We summarize basic properties in the following proposition-definition. 
\begin{proposition} \cite[Chapter 8]{gkz} \label{resultant}
Let $A_1,A_2,\ldots,A_k \subset \mathbb{Z}^{k-1}$ be a collection of $k$ finite sets. Then there exists a polynomial $R_{A_1,A_2,\ldots,A_k}$ with the following properties:
\begin{itemize}
    \item $R_{A_1,A_2,\ldots,A_k}$ has integral coefficients and it is irreducible.
    \item If $\Vector{f}=(f_1,f_2,\ldots,f_k)$ is a polynomial system with $f_i \in \C^{A_i}$ and $\Vector{f}$ has a zero in $(\C^{*})^{k-1}$, then $R_{A_1,A_2,\ldots,A_k}(\Vector{f})=0$.
\end{itemize}
\end{proposition}
An exposition for sparse resultants with a computational focus can be found in \cite{emiris}. A nice trick referred to as ``Cayley trick" relates $A$-discriminants and sparse resultants.
\begin{lemma} \cite[Chapter 9, Prop 1.7]{gkz} \label{cayleytrick}
Using the notation of \cref{resultant} and letting $\A:=A_1*A_2* \ldots *A_k$, we have
\[ R_{A_1,A_2,\ldots,A_k}(\Vector{f}(x)) = \Delta_{\A}(f_1(x) + \sum_{i=2}^{k} y_i f_i(x) ) \]
where $y_i$ denotes the new variables added in the construction of $\A$.
\end{lemma}

Now we would like to think about singular zeros of a sparse polynomial system. For a tuple of \struc{\textit{coefficient vectors}} $\struc{\CC}=(\CC_1,\CC_2,\ldots,\CC_k)$ with $\CC_i \in \mathbb{C}^{\# A_i}$, let $\p_{\CC}$ be the polynomial system $\struc{\p_{\CC}}=(p_1,p_2,\ldots,p_{k})$ with $\struc{p_i}=\sum_{\Vector{a}_{ij} \in A_i} \CC_{ij} \x^{\Vector{a}_{ij}} $. We define the \textit{\struc{discriminantal locus}} for systems of equations as follows:
\begin{align*} 
    \struc{\nabla_{A_1,A_2,\ldots,A_k}} \ := \  \ovl{  \left\{ (\CC_1,\CC_2,\ldots,\CC_k) \in \C^{A_1} \times \ldots \times \C^{A_k} : \p_{\CC}  \; \text{posses a singularity on} \;  (\C^{*})^n \right\} }.
\end{align*}
The toric variety that is dual to $\struc{\nabla_{A_1,A_2,\ldots,A_k}} $ may not be clear at first sight, but it is isomorphic to $X_{A_1+A_2+\ldots+A_k}$; we refer the reader to \cite[Chapter 8, Proposition 1.4]{gkz} for a nice explanation.

The discriminantal locus corresponding to hypersurfaces supported by the Cayley configuration $\A=A_1 * A_2 * \cdots * A_k$ is then given by
\begin{align*} 
    \struc{\nabla_{\A}} \ := \ \ovl{ \left\{ \CC \in \C^{\A} :  \sum_{a \in \A} c_a x^{a}  \; \text{posses a singularity on} \;  (\C^{*})^n \right\} }.
\end{align*}
If $\A=A_1 * A_2 * \cdots * A_k$ is not degenerate, then $\nabla_{\A}$ is an irreducible hypersurface. Also, by using the definition of singularity with the Jacobian matrix, it immediately follows that
$\nabla_{\A} \subseteq \nabla_{A_1,A_2,\ldots,A_n}  $. The following result of Esterov, proved by a simple perturbation argument, relates  $\nabla_{\A}$ and  $\nabla_{A_1,A_2,\ldots,A_k}$; see of \cite[Lemma 3.36]{esterov}, and note that in Esterov's notation $\nabla_{A_1,A_2,\ldots,A_k} $ is denoted by $ \Sigma_{A_0,A_1,\ldots,A_{\ell}}$.

\begin{theorem}[Esterov]
If $\A=A_1 * A_2 * \cdots * A_k$ is not defect, and $\dim (\conv(A_i))=n$ for $i=1,2,\ldots,k$, then $\nabla_{A_1,A_2,\ldots,A_k}$ is irreducible of codimension one. 
\end{theorem}

Hence, if the assumptions of Esterov's theorem are satisfied, then $\nabla_{\A}$ and  $\nabla_{A_1,A_2,\ldots,A_k}$ coincide. So, in order to control the changes in the topology for systems of equations supported with $A_1,A_2,\ldots,A_k$, we use the hypersurface $\nabla_{\A}(\R)$.
\subsection{Basics of amoeba theory} 
\label{subsection:amoeba}

In this section, we introduce the notion of amoeba following Gelfand, Kapranov, and Zelevinsky \cite{gkz}. For an overview of amoeba theory please see \cite{Mikhalkin:Survey,passare:tsikh}.

\begin{definition}
	\label{Definition:Amoeba}
	We define the \textit{\struc{Log-absolute value map}} as
    \begin{align*} 
        \struc{\Log}: (\C^{*})^{n} \to \R^{n}, \quad (z_1,z_2,\ldots,z_n) \to (\log \abs{z_1}, \log \abs{z_2}, \ldots, \log \abs{z_n}).
    \end{align*}
    For a Laurent-polynomial $f \in \C\left[\Vector{z}^{\pm 1}\right]$ and variety $\variety{f}\subset (\C^*)^n$ we define the \textit{\struc{amoeba}} of $f$ as $\struc{\Amoeba{f}} := \Log|\variety{f}| \subseteq \R^n$.
\end{definition}
\begin{lemma} \label{Lemma:AmoebaComponentsComplementVertices}
Let $f=\sum_{i} c_i \x^{\Vector{a}_i}$ be a polynomial with support  $A=\{\Vector{a}_1, \Vector{a}_2, \ldots, \Vector{a}_m \}$. Let $\Vector{v}$ be a vertex of $\New(f)$. Suppose that $\Vector{b} \in \NormalCone{\Vector{v}}$ with 
\begin{align*} 
    \langle \Vector{b} , \Vector{v}-\Vector{a}_i \rangle  \ > \ \log\left( \frac{m \cdot \abs{c_i}}{\abs{c_{\Vector{v}}}} \right)
\end{align*}
for all $\Vector{a}_i \neq \Vector{v}$. Then, $ \Amoeba{f} \cap \left( \Vector{b}+\NormalCone{\Vector{v}} \right)= \emptyset$. 
\end{lemma}

\begin{figure}[ht]
	\ifpictures
	\begin{center}
		\includegraphics[width=0.45\linewidth]{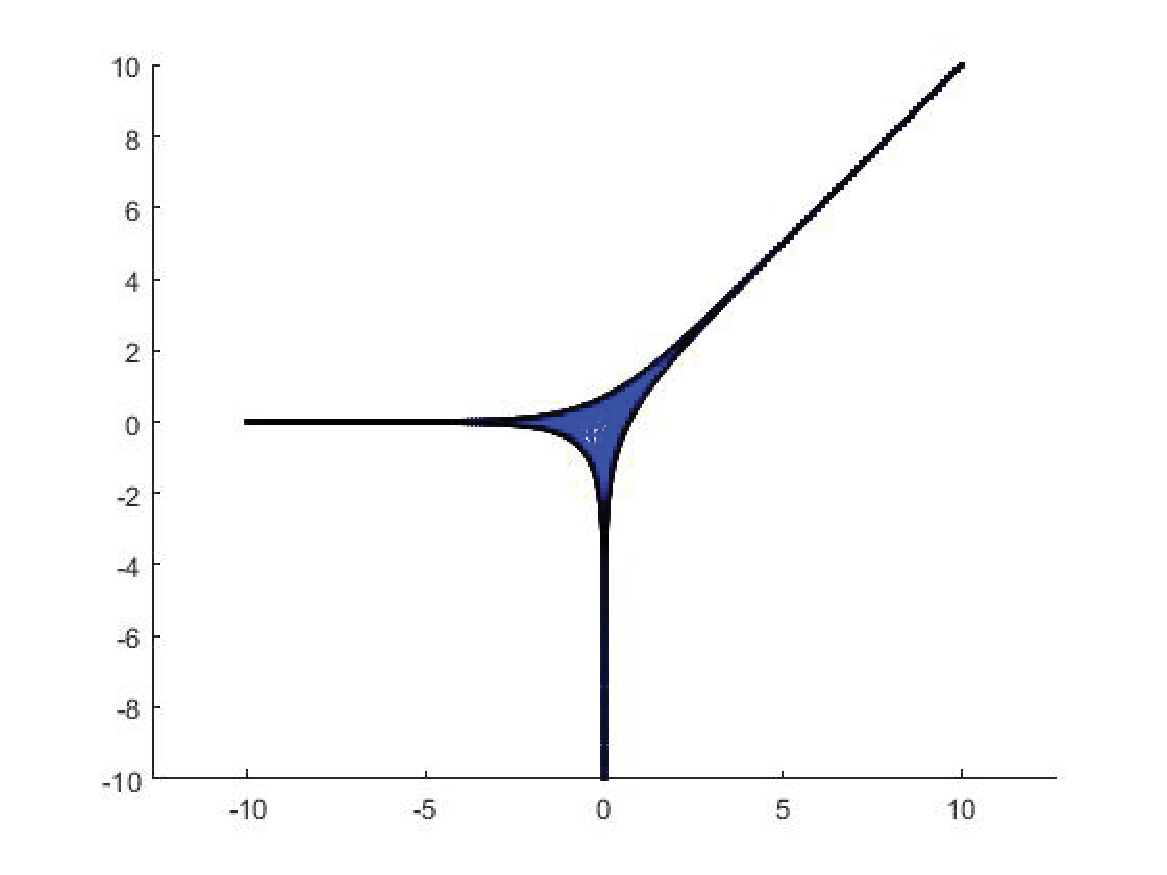} \qquad
		\includegraphics[width=0.45\linewidth]{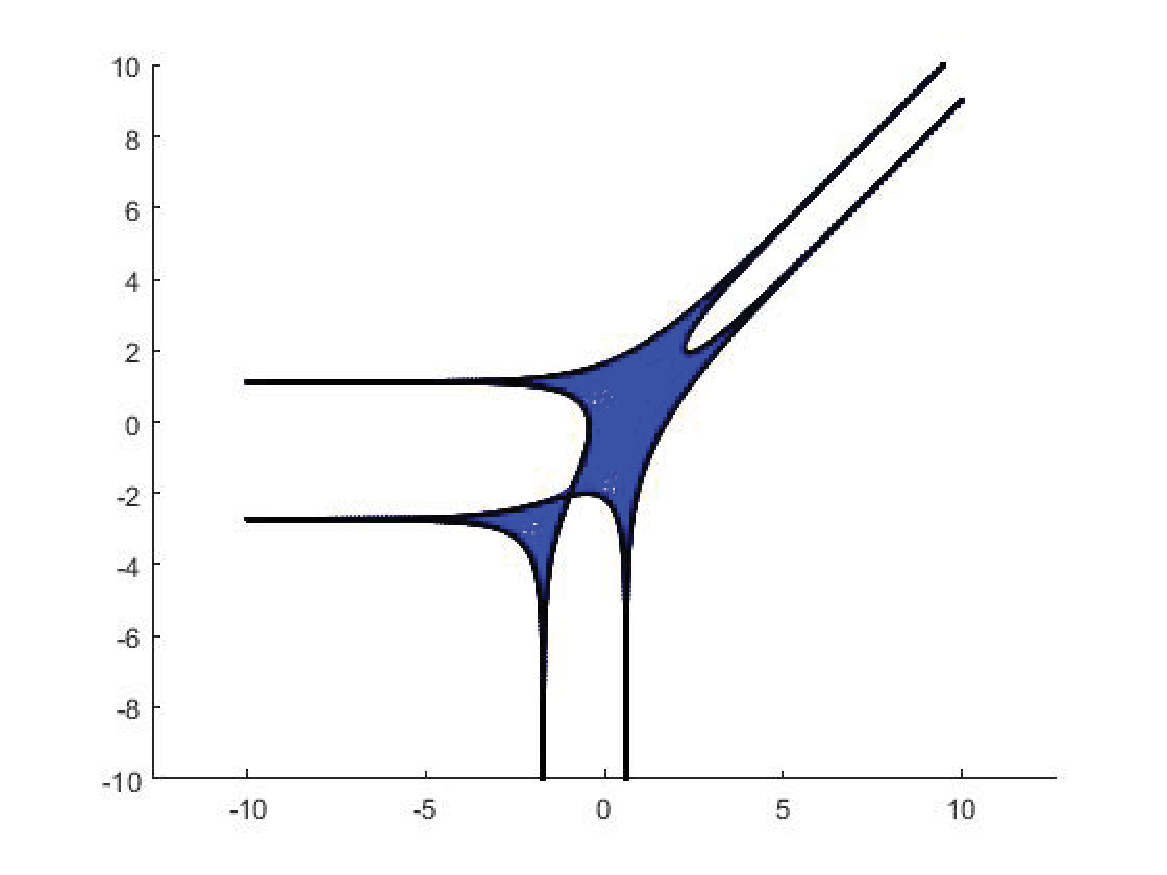}
	\end{center}
	\fi
    \caption{Amoebae for $x_1+x_2-1$ and $-1+5x_1-15x_2+10x_1x_2+3x_1^2+5x_2^2$ }
\end{figure} 

The statement is well-known; see \cite[Prop. 1.5, Page 195]{gkz}. Here, we provide the main argument of the proof for the convenience of the reader.

\begin{proof}
We have
\begin{align*} 
	f(\x) \ = \ c_{\Vector{v}} \x^{\Vector{v}} \left( 1 + \sum_{\Vector{a}_i \neq \Vector{v}} \frac{c_i}{c_{\Vector{v}}} \x^{\Vector{a}_i-\Vector{v}} \right).
\end{align*}
Set $g(\x)=\sum_{\Vector{a}_i \neq \Vector{v}} \frac{c_i}{c_{\Vector{v}}} \x^{\Vector{a}_i-\Vector{v}} $. Then for a given $\x \in (\C^{*})^{n}$ if $\abs{g(\x)} < 1$ this immediately implies $f(\x) \neq 0$ and hence $\Log|\x| \notin \Amoeba{f}$.  The rest of the proof is straightforward.
\end{proof}
\cref{Lemma:AmoebaComponentsComplementVertices} shows that for every $\Vector{v}$ of the $\New(f)$ there is an unbounded connected component in the complement of the amoeba $\Amoeba{f}$ that includes a copy of the normal cone $\NormalCone{b}$. The following are the basic facts: these connected components are distinct for every $\Vector{v}$, these connected components exhaust the list of unbounded components in the complement of $\Amoeba{f}$, and these connected components are convex \cite{Passare:Sadykov:Tsikh, passare:tsikh}.

\subsection{Real toric deformation} \label{subsection:mixedcell}
This section is to set up the real homotopy starting from combinatorial patchworking to our target system. We will require the deformation path to lie outside of a region given by a union discriminant amoebas. This will ensure there are no root paths that visit toric infinity and that the deformation preserves the geometry of the real zero set.
\begin{proposition} \label{toriclimit}
Let $A_1,A_2,\ldots,A_n \subset \Z^{n}$ be point configurations with  $\dim(A_i)=n$ for all $i \in [n]$, and let $\A=A_1 * A_2 * \cdots * A_n$ be the Cayley configuration.  
Suppose that $\Vector{v} = (\Vector{v}_{\Vector{a}})_{\{\Vector{a} \in A_i \ : \ 1 \leq i \leq n\}}  \in \R^{\A}$ and $\CC \in \mathbb{R}^{\A}$ are vectors with the following properties:
\begin{enumerate}
\item $\Vector{v}$ is not on the boundary of any secondary cone of the point configuration $\A$.
\item For every face $\Gamma$ of $\A$, except for the irrelevant faces,  the ray $\Log|\CC|+ \lambda  \Vector{v}$ for $\lambda \in [0,\infty)$ does not intersect the amoeba of $\Delta_{\Gamma}(\R)$.
\end{enumerate}
We consider a system of equations $\p_{\CC}(t,\x)=(p_1,p_2,\ldots,p_n)$:
\begin{equation} \label{tr1}
p_i(t,\x) \ = \ \sum_{\Vector{a} \in A_i} c_{\alpha} t^{- v_{\Vector{a}}} \x^{\Vector{a}}  \ \text{ for } \ i=1,2,\ldots,n .
\end{equation}
Then the real Puiseux series 
\begin{equation} \label{tr2}
\x(t)= (x_1 t^{\zeta_1}, x_2 t^{\zeta_2}, \ldots, x_n t^{\zeta_n}) + \; \text{higher order terms} 
\end{equation}
is a solution to the system $\p_{\CC}$  only if $(\Vector{\zeta},1)$ is an outer normal to a lower facet of 
\begin{align*} \conv(A_1^{\Vector{v}} + A_2^{\Vector{v}} + \ldots + A_n^{\Vector{v}})  \end{align*}
where $\struc{A_i^{\Vector{v}}}$ stands for the lifting of $A_i$ with respect to  $\Vector{v} \in \mathbb{R}^{\A}$. Moreover, $\varietyrealnonzero{p_i(t,x)}$ are isotopic for all $t \in (0,1]$.
\end{proposition}
\begin{remark}
We note that  in the statement we have $\Log|\CC|+ \lambda  \Vector{v}$ for $\lambda \in [0,\infty)$, and also $c_{\alpha} t^{-v_{\alpha}}$ for $t \in (0,1]$; these two represent the same parameter regime as $\log \abs{c_{\alpha} t^{-v_{\alpha}}} = \log{\abs{c_{\alpha}}} + v_{\alpha} \log \frac{1}{t}$. 
\end{remark}
\begin{proof}
The statement about the Puiseux series follows the same proof as \cite[Lemma 3.1]{huberpolyhedral}, so we just list the main steps: Put \cref{tr2} into \cref{tr1}, divide by the lowest degree term, and set $t=0$. The system of equations obtained this way will have at most $2n$ terms in total, and it can have a common zero only if it is a system of binomial equations. On can observe that under $\Log$-map the solutions of these binomial equations correspond to the finite number points given by Viro's method. We had already discussed in \cref{subsection:anders} and \cref{subsection:binomials} that these points given by Viro's method identify alternating mixed cells.

Now we consider the statement about isotopy. Let $S := \{ x \in \mathbb{R}^{2n-1} : x_{n+1}= x_{n+2} = \ldots = x_{2n-1} = \frac{1}{n-1} \}$. We recall that $\Cayley(\A) \cap S$ and $A_1+A_2+\ldots+A_n$ are equivalent up to scaling, see \cref{subsection:polyhedralcayley}. Let $\Gamma$ be a proper face of $\A$ that is not irrelevant, let $\tilde{\Gamma}$ be the face of $A_1+A_2+\ldots+A_n$ equivalent to $\Gamma \cap S$, and suppose $\tilde{\Gamma}= \tilde{\Gamma}_1 + \tilde{\Gamma}_2 + \ldots + \tilde{\Gamma}_n$ where $\tilde{\Gamma}_i$ is a face of $A_i$. Observe that $\Gamma =  \tilde{\Gamma}_1 * \tilde{\Gamma}_2 * \ldots * \tilde{\Gamma}_n$. By  \cref{cayleytrick}, if a polynomial system $\Vector{f|_{\Gamma}}$ satisfies $\Delta_{\Gamma}(\Vector{f|_{\Gamma}}) \neq 0$ then $\Vector{f}$ has no zero on $X_{\tilde{\Gamma}}^{\circ}$. Thus if $\Delta_{\Gamma}(\Vector{f|_{\Gamma}}) \neq 0$ for all proper faces, except the irrelevant ones, then $\Vector{f}$ has no zero on $X_{A_1+A_2+\ldots+A_n}-X_{A_1+A_2+\ldots+A_n}^{\circ}$. 

The ray $\Log|\CC|+ \lambda  \Vector{v}$ does not intersect the amoeba of $\Delta_{\Gamma}(\R)$ for any $\lambda \in [0,\infty)$. This implies $\Delta_{\Gamma}(\Vector{\Vector{p_{\CC}(t)}}) \neq 0$ for all faces $\Gamma$, except irrelevant ones, and also $\Delta_{\A}(\Vector{\Vector{p_{\CC}(t)}}) \neq 0$ for all $t \in (0,1]$. So the condition $\Delta_{\Gamma}(\Vector{\Vector{p_{\CC}(t)}}) \neq 0$ guarantees non-existence of zeros on $X_{A_1+A_2+\ldots+A_n}-X_{A_1+A_2+\ldots+A_n}^{\circ}$, and since $\Delta_{\A}(\Vector{\Vector{p_{\CC}(t)}}) \neq 0$ we have that the zero sets are isotopic on $X_{A_1+A_2+\ldots+A_n}^{\circ}$. Thus, $\varietyrealnonzero{p_i(t,x)}$ are isotopic for all $t \in (0,1]$.
\end{proof}
\vspace{-0.1 in}
\subsection{Numerically tracking a solution from toric infinity}
\label{subsection:pathtrackers}

The numerical part of our algorithm tracks real zeros of $\p_{\CC}(t,\x)$, as in \cref{toriclimit}, from  $\p_{\CC}(0,\x)$ to
$\p_{\CC}(1,\x)$. There are several technicalities to be careful about: 1-) we are not able to start the homotopy continuation precisely at $\p_{\CC}(0,\x)$ since all its zeros lie at toric infinity, 2-) we need to design an algorithm to track the solution paths  $\x(t)$, as in \cref{toriclimit}, from $t \sim 0$ to $t=1$. 

The first issue is theoretically handled by an analytic continuation argument on toric compactification, and it is practically handled by predictor-corrector methods in numerical analysis. The second part, tracking the solution paths, can be done in two ways: 
\begin{enumerate}
	\item trace the solution curves $\x(t)$ numerically, or
	\item start a homotopy from $\p_{\CC}(0,\x)$ with zeros given 
by alternating mixed cells and track the solution path from $t=0$ to $t=1$.
\end{enumerate}
Explaining details of these numerical schemes have the potential of doubling the size of our paper and the techniques are now folklore, so we prefer to have a brief account. An established reference for curve tracing approach, i.e., the first method, is \cite{numericalcont}. The curve tracing approach is often fast, and it is a standard technique in numerical analysis that is deployed in many applications. However, to the best of our knowledge, the safeguards to control precision issues for standard path trackers only exist for specific cases. The second approach  has a well developed theory to control precision issues and conduct rigorous complexity analysis  in the case of dense polynomials \cite{condition}. For sparse polynomials, Malajovich recently developed a theory that allows to express complexity of numerical tracking with certain integrals of condition numbers \cite{m2}. We briefly explain  Malajovich's approach in \cref{section:numericalcomplexity}. Our algorithm can be implemented using any of the two ways depending on the preferred trade-off between rigor and speed. For a nice exposition on comparing the two alternatives  we suggest \cite[Section 2.3 and 2.4]{bates-sottile}. 
\subsection{An entropy type formula for the discriminant locus} 
\label{subsection:Horn-Kapranov}
In this section, we introduce useful facts about $A$-discriminants, mostly relying on \cite[Chapter 9, Section 3, subsection C]{gkz} and works of Passare and Tsikh \cite{passare:tsikh}.

\begin{theorem}[Horn-Kapranov Uniformization] \label{H1}
Let $A=[ \Vector{a}_1, \Vector{a}_2, \ldots, \Vector{a}_m ]$ be a collection of lattice points in $\Z^n$, let $\nabla_A$ be the corresponding $A$-discriminant variety. 
We consider $A$ as a $n \times m$ matrix, and define
\begin{align*} 
	\struc{\Psi_A(\Vector{u},\x)} \ := \ \left[u_1 \x^{a_1} : u_2 \x^{a_2} : \ldots : u_m \x^{a_m} \right].
\end{align*}
Then $\nabla_A$ admits the following parametrization:
\begin{align*} 
	\nabla_A \ = \ \ovl{ \left\{   \Psi_A(\Vector{u},\x) \ : \ A\Vector{u}=\Vector{0} , \sum_{i = 1}^m u_i =\Vector{0}, \x \in (\C^{*})^n \right\} }.
\end{align*}
\end{theorem}
Now consider the amoeba of $\nabla_A$:
\begin{align*}
	\Log|\nabla_A| \ = \ \Log\left| \left\{ \Vector{u} \ : \  A\Vector{u}=\Vector{0} , \sum_{i = 1}^m u_i =\Vector{0} \right\}\right| \ + \ \left( \Log|\x| \right)^T A
\end{align*}
where $+$ denotes the Minkowski sum. 

It is easy to observe that $\left( \Log|\x| \right)^T A$ corresponds to the row span of $A$. Moreover, for any $\Vector{u}$ with $A\Vector{u}=\Vector{0}, \sum_{i = 1}^m u_i =\Vector{0}$ any scalar multiple of $\Vector{u}$ satisfies the same equations. This $n$-dimensional row span and one dimensional linear space represents $n+1$ homogeneities that are present in the discriminant variety; the variety is invariant under torus action and scaling.

For a given hypersurface $\variety{f} \subseteq (\C^*)^n$ consider all points which are critical under the $\Log|\cdot|$ map. The $\Log|\cdot|$-image of these points is called the \textit{\struc{contour}} of the corresponding amoeba $\Amoeba{f}$; see e.g., \cite{passare:tsikh}. It is straightforward to show that the contour contains the boundary $\partial \Amoeba{f}$, but does not coincide with it in general; see e.g., \cite{passare:tsikh}.  Moreover, for a real polynomial $f$, the contour contains the amoeba of the smooth part of the real variety, i.e. $\Amoeba{\varietyrealnonzero{f}}$ \cite{passare:tsikh}.

Let $\struc{B}$ be a \textit{\struc{Gale dual}} of $A$, i.e, an $m \times (m-n-1)$ integer matrix that has all column sums to be $0$ and satisfies  $AB=\Vector{0}$. Then, for any $\Vector{u} \in (\R^{*})^{m}$ with $A \Vector{u} = \Vector{0}$ and $\sum_i u_i = 0$ one can find a $\Vector{\zeta} \in (\R^{*})^{m-n-1}$ with $\Vector{u}=B \Vector{\zeta}$.  

It follows from the discussion in \cite{passare:tsikh} (see the section titled Discriminants and Real Contours, and specifically Theorem 4), that the parametrization of the contour of  the reduced $A$-discriminant amoeba $B^{T}\Amoeba{\nabla_A(\C)}$ is given as follows:
\begin{align} \label{H1}
 B^T \Log\left| \left\{ \Vector{u} : \Vector{u} \in (\R^{*})^m  , A \Vector{u} = \Vector{0} , \sum_i u_i = \Vector{0} \right\}\right|.  
 \end{align}
Using \cref{H1} and the fact that contour includes the amoeba of the real part of the variety, one can concisely write
\begin{equation} \label{H2}
B^{T} \Amoeba{\nabla_A(\R)} \ \subseteq \ \left\{ B^{T} \Log|\Vector{u}| \ : \  \Vector{u} \in (\R^{*})^m, A \Vector{u} = \Vector{0} , \sum_{i = 1}^m u_i = 0 \right\}. 
\end{equation}

Using the row space of $B$ to parameterize the set $\{ \Vector{u} \in (\R^{*})^m, A \Vector{u} = \Vector{0} , \sum_{i = 1}^m u_i = 0 \}$, this can also be written as follows:
\begin{equation} \label{H4}
B^{T} \Amoeba{\nabla_A(\R)}  \ \subseteq \ \left\{ \sum_{i = 1}^m \Vector{b(i)} \log \abs{ \langle \Vector{b(i)} , \Vector{\zeta} \rangle} \ : \ \Vector{\zeta} \in (\R^{*})^{m-n-1}  \right\}
\end{equation}
where the $\Vector{b(i)}$ denote the rows of $B$. As a next step, we define the following map:
\begin{align*} 
	\struc{\phi_A}: (\R^{*})^{m-n-1}  \rightarrow (\R^{*})^{m-n-1} \; \; , \; \;  \phi_A(\Vector{\zeta}) =  \sum_{i = 1}^m \Vector{b(i)} \log \abs{ \langle \Vector{b(i)} , \Vector{\zeta} \rangle}.   
\end{align*}
The facts listed follows from \cite[Chapter 9, Section 3, subsection C]{gkz}:
\begin{enumerate} \label{horn}
\item The map $\phi_A$ is $0$-homogeneous, that is for every $\lambda \in (0,\infty)$ and $\Vector{\zeta} \in (\R^{*})^{m-n-1}$ we have
\begin{align*} 
	\phi_A(\lambda \Vector{\zeta}) \ = \ \phi_A(\Vector{\zeta}).
\end{align*} 
\item The image of the map $\phi_A$ is a hypersurface, and if the \struc{\textit{Gauss map} $\gamma$} is defined at $\phi_A(\Vector{\zeta})$ then we have 
\begin{align*} 
	\gamma( \phi_A(\Vector{\zeta}) ) = \Vector{\zeta}.
\end{align*}
\end{enumerate}
The first property follows since the column sums of $B$ equals $0$. The second property is proved by Kapranov \cite{kapranov91}. Now assume that we have a $\Vector{\zeta} \in (\R^{*})^{m-n-1}$, and we would like to write down the equation of the tangent hyperplane $\struc{H_{\Vector{\zeta}}}$ at $\phi_A(\Vector{\zeta})$. 

Since we know the image under the Gauss map (i.e., the normal direction), we obtain:
\begin{align*} 
	H_{\Vector{\zeta}} \ = \ \left\{ \x \in \R^{m-n-1} : \langle x , \Vector{\zeta}  \rangle= \langle \phi_A(\Vector{\zeta}) , \Vector{\zeta} \rangle  \right\}.
\end{align*}
One can rewrite this as follows:
\begin{equation} \label{hyperplane-A}
 H_{\Vector{\zeta}} \ = \ \left\{  \x \in \R^{m-n-1} : \langle \Vector{\zeta} , \x \rangle = \sum_{i = 1}^m  \langle \Vector{b(i)}, \Vector{\zeta}  \rangle \log \abs{ \langle \Vector{b(i)} , \Vector{\zeta} \rangle}  \right\}.
\end{equation}


\section{Effective Viro's Patchworking} \label{amoebamagic}
Consider a polynomial system $\Vector{p}=(p_1,p_2,\ldots,p_n)$ with support sets $A_1,A_2,\ldots,A_n$ and the coefficient vector $\CC=(\CC_1,\CC_2, \ldots, \CC_n)$. How do we decide if the common real zero set of $\Vector{p}$  (up to continuous deformation) can be described by Viro's patchworking method? 
Here we present a way to certify if this is the case for a given system $\Vector{p}$: We search for a ray $\Log|\CC| + \lambda \Vector{v}$  that  does not intersect the discriminant amoebae for all faces $\Gamma$ of $\A$, except the irrelevant ones. This represents a real toric deformation starting from $\Vector{p}$ as in \cref{toriclimit}. 

\subsection{Statement of main result and an example}
We keep the notation from \cref{subsection:Horn-Kapranov}, what follows is the main result of this section. 

\begin{proposition} \label{Proposition:MainCertificateOurAlgorithm}
Let $\Vector{p}_{\CC}$ be a system of sparse polynomials with coefficient vector $\CC$ and support sets $A_1,A_2,\ldots,A_n \subset \Z^n$ where $\dim (\A_i)=n$ for all $1 \leq i \leq n$. Let $T$ be the triangulation of the Cayley configuration $\A=A_1*A_2*\ldots*A_n$ that is introduced by using $\Log \CC $ as a lifting function. Let $M(T)$ be the corresponding mixed cell cone, and suppose that the dual cone $M(T)^{\circ}$ is generated by vectors $\Vector{\zeta(1)},\ldots,\Vector{\zeta(L)}$. Then, if 
\begin{equation} \label{crt}
 \langle \Log \CC , \Vector{\zeta(i)} \rangle \ > \ \log(\# \A) \norm{\Vector{\zeta(i)}}_1  
\end{equation}
for all $i=1,2,\ldots,L$, the system $\Vector{p}_{\CC}$ is a patchworked polynomial system. Furthermore, for any $\Vector{v} \in M(T)$ the ray $\Log \CC + \lambda \Vector{v}$ for $\lambda \in  [0,\infty)$ does not intersect the amoeba of $\Delta_{\Gamma}$ for all faces $\Gamma$ of $\A$ except for the irrelevant ones.
\end{proposition}

Note that, for any coefficient vector $\CC$ with corresponding triangulation $T$, if the generators of the dual mixed cell cone $M(T)^{\circ}$ are $\Vector{\zeta(i)}$ for $i=1,2,\ldots,L$, then, by definition,
\begin{align*}
	\langle \Vector{\zeta(i)} , \Log \CC \rangle \ > \  0
\end{align*}
for all $i=1,2,\ldots,L$. To apply \cref{Proposition:MainCertificateOurAlgorithm}, we need
\begin{align*}
	\langle \Vector{\zeta(i)}, \Log \CC  \rangle \ > \ \log(\# \A) \norm{\Vector{\zeta(i)}}_1.
\end{align*}
Here $\norm{\Vector{\zeta(i)}}_1 $ is a normalization; one can just use normalized generators with unit $\ell_1$-norm. So the loss in our relaxation is represented by the logarithmic term $\log(\# \A)$.

Let us illustrate \cref{Proposition:MainCertificateOurAlgorithm} on the simplest case: univariate polynomials. Let $A=\{ 0 , a_1, a_2, \ldots, a_{2d} \} \subset \mathbb{Z}$, and let $p(x)=c_0 + c_1 x^{a_1} + c_2 x^{a_2} + \ldots + c_{2d} x^{a_{2d}}$. Here a triangulation is subdivision of the interval $[0,a_{2d}]$ into a union of smaller sub-intervals $[a_i,a_j]$. Suppose the lifting function $\Log C =(\log \abs{c_0}, \log \abs{c_1}, \ldots , \log \abs{c_{2d}})$ introduces the triangulation  $T = \{ [0,a_2], [a_2, a_4], \ldots, [a_{2(d-1)},a_{2d}] \}$. The first ``simplex'' being $[0,a_2]$ means for every $a_i$ with $i \neq 2$ we must have $(a_i, \log \abs{c_i})$ lying above the line segment $\{ (0, \log \abs{c_0}), (a_2, \log \abs{c_2}) \}$. In terms of circuit inequalities,  this means the following
\[ \frac{\log \abs{c_2} -\log \abs{c_0}}{a_2} < \frac{\log \abs{c_i} - \log \abs{c_0}}{a_i} \; \;  \text{for} \; i=1,3,4,5,\ldots, d \]
Or, equivalently
\[ \log \abs{c_0} (a_i-a_2) - \log \abs{c_2} a_i  + \log \abs{c_i} a_2  > 0  \; \;  \text{for} \; i=1,3,4,5,\ldots, d \]
The hypothesis of  \cref{Proposition:MainCertificateOurAlgorithm}  amounts to 
\[ \log \abs{c_0} (a_i-a_2) - \log \abs{c_2} a_i + \log \abs{c_i} a_2 > \log (d+1) (\abs{a_i-a_2} + a_i + a_2). \]

If the hypothesis of \cref{Proposition:MainCertificateOurAlgorithm} is satisfied for all the circuit inequalities of the triangulation $T$ (that is all generators of $M(T)^{\circ}$) then the number of real zeros of $p(x)$ can be counted as follows:  Let $\sgn c_i$ represent the signs of $c_i$, set
\[ \sgn C :=( \sgn c_0, \sgn c_2, \sgn c_4, \ldots, \sgn c_{2d} ), \]
and let $k$ be the number of sign changes in the vector $\sgn C$. The vector $\sgn C$ represents the signs relevant to the triangulation $T$, and due to nature of $T$ we have the same sign vector on``negative orthant" $(-\infty,0)$. Then \cref{toriclimit} combined with \cref{Proposition:MainCertificateOurAlgorithm} says $p$ has $2k$ many real zeros.
\subsection{Some basic results on the complement of $A$-discriminant amobea}
For simplicity, we let $m= \# \A$. Note that $\A \subset \mathbb{Z}^{2n-1}$. Here we assume the reader is familiar with the basic facts from \cref{subsection:amoeba} and start with the following lemma. 
\begin{lemma} \label{Lemma:convex}
Let $\Vector{\eta} $ be a vertex of the Newton polytope  of $\Delta_{\A}$, and let $K_{\Vector{\eta} }$ be the corresponding connected component in the complement of the $\Vector{A}$-discriminant amoeba. 
\begin{enumerate}
\item Let $\Vector{u} \in K_{\Vector{\eta} }$ and let $\Vector{v} \in \NormalCone{\Vector{\eta}}$ then the ray $\Vector{u}+ \lambda \Vector{v}$ for $\lambda \in [0,\infty)$ does not intersect the $\Vector{A}$-discriminant amoeba. 
\item Let $\Phi_{\A}$ and $B$ respectively be the map and the matrix defined in \cref{subsection:Horn-Kapranov}. Suppose $\Vector{\zeta} \in \mathbb{R}^{m-2n}$ with $\Phi_\A(\Vector{\zeta}) \in \partial(B^{T} K_{\Vector{\zeta}})$, then $\Vector{\zeta} \in \left( B^{T} \NormalCone{\Vector{\eta}} \right)^{\circ}$.   
\end{enumerate} 
\end{lemma}

\begin{proof}
As $K_{\Vector{\eta} }$ is a component of the complement of an amoeba, it is a convex set. 
Moreover, by \cref{Lemma:AmoebaComponentsComplementVertices}, it includes a shifted copy of $\NormalCone{\Vector{\eta}}$. 
Now let  $H_{\w} := \{ \langle \w , \x \rangle =  c \} $ be a supporting hyperplane of $K_{\Vector{\eta} }$ (i.e., for every $\Vector{y} \in K_{\Vector{\eta} }$ we have $\langle \w , \Vector{y} \rangle \geq c$). We claim  $\w \in \NormalCone{\Vector{\eta}}^{\circ}$: otherwise the shifted copy of the cone $\NormalCone{\Vector{\eta}}$, that is included in $K_{\Vector{\eta} }$, would intersect the supporting hyperplane $H_{\w}$, which is a contradiction. 

Let $\Vector{u} \in K_{\Vector{\eta} }$ and $\Vector{v} \in \NormalCone{\Vector{\eta}}$. 
Then we have for any $\w \in \NormalCone{\Vector{\eta}}^{\circ}$ and $\lambda > 0$
\begin{align*}
	\langle \w , \Vector{u} \rangle \ \leq \ \langle \w , \Vector{u} + \lambda\Vector{v} \rangle.
\end{align*}
Hence, the ray $\Vector{u} + \lambda \Vector{v} $ does not intersect any supporting hyperplane of $K_{\Vector{\eta} }$, and in consequence does not intersect the boundary of the convex set $K_{\Vector{\eta} }$.

Now suppose that we have a $\Vector{\zeta} \in \mathbb{R}^{m-2n}$ with $\Phi_\A(\Vector{\zeta}) \in \partial(B^{T} K_{\Vector{\zeta}})$, then by the second property in \cref{horn} the supporting hyperplane at $\Phi_\A(\Vector{\zeta})$ will be
\begin{align*}
	H_{\Vector{\zeta}} \ := \ \left\{  \x \in \R^{m-2n} \ : \ \langle \Vector{\zeta} , \x \rangle = \sum_{i}^m  \langle \Vector{b(i)}, \Vector{\zeta}  \rangle \log \abs{ \langle \Vector{b(i)} , \Vector{\zeta} \rangle}  \right\}. 
\end{align*}
Since there is a shifted copy of $B^{T} \NormalCone{\Vector{\eta}}$ inside the convex set $B^{T} K_{\Vector{\eta} }$, this shows that $\Vector{\zeta} \in \left( B^{T}   \NormalCone{\Vector{\eta}} \right)^{\circ}$. 
\end{proof}

The Gale dual matrix $B$ in \cref{Lemma:convex} is of size $m \times (m-2n)$. Thus, $B^{T} K_{\Vector{\zeta}}$ is a projection of $K_{\Vector{\eta} }$ from $\R^m$ to $\R^{m-2n}$. The kernel of the matrix $B^{T}$ is included in every connected component $K_{\Vector{\eta} }$ of the complement of the $\A$-discriminant amoeba and this projection creates no loss of generality. There are two ways to see this: Formally the kernel of $B^T$ is the span of rows of $A$ and $(1,1,\ldots,1)$ vector, and in \cref{theorem:secondarypolytope} it was noted that this space is included in every secondary cone and hence also in $\NormalCone{\Vector{\eta}}$ by \cref{Lemma:Mixed-cell-cone}.  Geometrically, the kernel of $B^T$ represents the homogeneities present in the $\Vector{A}$-discriminant variety as explained in \cref{subsection:Horn-Kapranov}.

Given a point $\Log|\CC|$, testing if $\Log|\CC| \in K_{\Vector{\eta} }$ is equivalent to testing if $B^{T} \Log|\CC|  \in B^{T} K_{\Vector{\eta} }$; the kernel of $B^{T}$ is included in all $K_{\Vector{\eta} }$. One can test whether $B^{T} \Log|\CC|  \in B^{T} K_{\Vector{\eta} }$ by checking all the supporting hyperplanes of $B^{T} K_{\Vector{\eta} }$ due to convexity. By \cref{Lemma:convex} and the discussion in \cref{subsection:Horn-Kapranov} we know that these supporting hyperplanes are of the form
\begin{align*}
	H_{\Vector{\zeta}} \ := \ \left\{  \x \in \R^{m-2n} : \langle \Vector{\zeta} , \x \rangle = \sum_{i}^m  \langle \Vector{b(i)}, \Vector{\zeta}  \rangle \log \abs{ \langle \Vector{b(i)} , \Vector{\zeta} \rangle}  \right\}
\end{align*}
for some $\Vector{\zeta} \in \left( B^{T} \NormalCone{\Vector{\eta}} \right)^{\circ}$. Now let $T$ be a triangulation of $\A$, and let $\Vector{\eta} $ be a vertex in the Newton polytope of $\Delta_{\A}$ with the property 
\begin{align*}
	\NormalCone{T} \ \subseteq \ M(T) \ \subseteq \ \NormalCone{\Vector{\eta}},
\end{align*}
see \cref{Lemma:Mixed-cell-cone}. By linearity, this means
\begin{align*}
	& B^{T}\NormalCone{T} \ \subseteq \ B^{T}M(T) \ \subseteq \ B^{T}\NormalCone{\Vector{\eta}}, \ \text{ and } \\
	& \left( B^{T}\NormalCone{\Vector{\eta}} \right)^{\circ} \ \subseteq \ \left( B^{T} M(T) \right)^{\circ} \ \subseteq \ \left( B^{T} \NormalCone{T} \right)^{\circ}.
\end{align*}
Instead of checking hyperplanes defined by $ \Vector{\zeta} \in \left( B^{T} \NormalCone{\Vector{\eta}} \right)^{\circ}$ we check the inequalities given by the larger cone $\left( B^{T} M(T) \right)^{\circ}$.  Before we explain the reason for this, we make an observation: $B(B^{T} M(T))^{\circ} \subseteq M(T)^{\circ}$ as follows:
\begin{align*}
	\x \in (B^{T} M(T))^{\circ} \ \Rightarrow \ \langle B \x , \Vector{y} \rangle \geq 0 \; \text{for all} \; \Vector{y} \in M(T).
\end{align*}
Also note that by definition we have
\begin{align*}
	\langle \Vector{\zeta} , B^{T} \Log \CC \rangle \ = \ \langle B \Vector{\zeta}, \Log \CC \rangle.
\end{align*}
So instead of using $\langle \Vector{\zeta} , B^{T} \Log \CC \rangle > 0$ for $ \Vector{\zeta} \in \left( B^{T} \NormalCone{\Vector{\eta}} \right)^{\circ}$  as our criterion we will use  $ \tau, \Log \CC \rangle > 0$ for all $\tau \in M(T)^{\circ}$. There are two  reasons for this: First reason is that to ensure the criterion in \cref{toriclimit} is satisfied we indeed have to check with all circuit inequalities in  $M(T)^{\circ}$. \cref{toriclimit}  involves amoebae of all $\Delta_{\Gamma}$ for all faces of $\A$, except the irrelavant ones, and checking with the dual cones coming from all such $\Gamma$ is equivalent to checking all inequalities in $M(T)^{\circ}$. This first reason will become more clear in \cref{finalshowdown}. The second reason is algorithmic efficiency: we had already computed the generators of $M(T)^{\circ}$ along the way, these are the circuit inequalities computed by Jensen's tropical homotopy algorithm Hence using $M(T)^{\circ}$ does not yield a significant computational cost.
\subsection{Quantitative Estimates}
\begin{lemma} \label{Lemma:Anders}
Let $T$ be a triangulation of $\A$, please keep the notation from \cref{Lemma:convex} for $K_{\Vector{\eta}}$ and $B$. If a given vector $\Log|\CC|$ satisfies
\begin{align*}
	\left\langle \Vector{\zeta}, B^{T} \Log \CC \right\rangle \ > \ \log(m) \norm{B \Vector{\zeta}}_1
\end{align*}
for all $\Vector{\zeta} \in (B^{T}M(T))^{\circ}$, then we have $\Log|\CC| \in K_{\Vector{\eta}}$ for a vertex $\Vector{\eta}$ of $\Delta_\A$ which satisfies 
$M(T) \subseteq \NormalCone{\Vector{\eta}}$. 
\end{lemma}

The proof of \cref{Lemma:Anders} will follow after we make some observations. We first note a basic observation on entropy type sums. 
\begin{lemma} \label{nice}
Let $\x \in \mathbb{R}_{\geq 0}^{d}$ be a vector with nonnegative entries. Then, we have
\begin{align*}
	\norm{\x}_1 \log \norm{\x}_1 - \log(d) \norm{\x}_1 \leq \sum_{i=1}^d x_i \log (x_i)  \leq \norm{\x}_1 \log \norm{\x}_1,
\end{align*}
where $\norm{\x}_1=\sum_{i=1}^d |x_i|$ represents the $\ell_1$-norm of the vector $\x$. 
\end{lemma}

\begin{proof}
Let $\Vector{y}:=\frac{\x}{\norm{\x}_1}$. Since $\norm{\Vector{y}}_1=1$, and it has nonnegative entries, we can see $\Vector{y}$ as a discrete probability distribution supported on $d$ strings. As usual
$H(\Vector{y})=\sum_{i=1}^d -y_i\log(y_i) $ is the entropy of $\Vector{y}$, and it is well-known that $H(\Vector{y}) \leq \log(d)$ \cite{codingtheory}. So, we have
\begin{align*}  
	H(\Vector{y}) \ = \ \frac{1}{\norm{\x}_1}\left( \sum_{i=1}^d x_i \log\norm{\x}_1 - x_i \log(x_i) \right) \leq \log(d).  
\end{align*}
This gives us the following inequality
\begin{align*} 
	\log \norm{\x}_1 \sum_{i=1}^d x_i \ \leq \ \log(d) \norm{\x}_1 +  \sum_{i=1}^d x_i \log (x_i),
\end{align*}
which proves the left-hand side inequality in the claim. The right-hand side is obvious.
\end{proof}

Now we derive the following useful estimate based on \cref{nice}.

\begin{lemma} \label{usefulestimate}
Let $\A$ be the support set, and let $B$ be the $ m \times (m-2n)$ Gale dual. Then, for every $\Vector{\zeta} \in \R^{m-2n}$ we have
\begin{align*} 
	-\frac{1}{2} \norm{B \Vector{\zeta}}_1 \log(m) \ \leq \ \sum_{i=1}^m \langle \Vector{b(i)} , \Vector{\zeta} \rangle \log \abs{\langle \Vector{b(i)} , \Vector{\zeta} \rangle}   \ \leq \  \frac{1}{2} \norm{B \Vector{\zeta}}_1 \log(m).
\end{align*}
\end{lemma}

\begin{proof}
By construction, every element in the column space of $B$ has the sum of its coordinates equal to zero. So, for every $\Vector{\zeta} \in \mathbb{R}^{m-2n}$ the sum of the entries of 
$B \Vector{\zeta}$ is zero. That is,
\begin{align*}  
	\sum_{i=1}^m \langle \Vector{b(i)} , \Vector{\zeta} \rangle \ = \ \mathrm{0},
\end{align*}
where $\Vector{b(i)}$ represents rows of the matrix $B$. We write $B \Vector{\zeta} = (\x,-\Vector{y})$ for some $\x$ and $\Vector{y}$ that are nonnegative in all coordinates, so we have $\norm{\x}_1=\norm{\Vector{y}}_1=\frac{1}{2}\norm{B\Vector{\zeta}}_1$.  We also observe
\begin{align*} 
	\sum_{i=1}^m \langle \Vector{b(i)} , \Vector{\zeta} \rangle \log \abs{\langle \Vector{b(i)} , \Vector{\zeta} \rangle} \ = \ \sum_{i=1}^{m_1} x_i \log(x_i) - \sum_{i=1}^{m_2} y_i \log(y_i).
\end{align*}
Note that  $m_1$ and $m_2$ in the above expression are both less than $m$. Using \cref{nice} and $\norm{\x}_1=\norm{\Vector{y}}_1=\frac{1}{2}\norm{B\Vector{\zeta}}_1$ gives us the following estimate:
\begin{equation} 
 - \frac{1}{2}\norm{B \Vector{\zeta}}_1 \log(m) \ \leq \ \sum_{i=1}^m \langle \Vector{b(i)} , \Vector{\zeta} \rangle \log \abs{\langle \Vector{b(i)} , \Vector{\zeta} \rangle}   \ \leq \ \frac{1}{2} \norm{B \Vector{\zeta}}_1 \log(m).
\end{equation}
\end{proof}

\begin{proof}[Proof of \cref{Lemma:Anders}]
Using \cref{usefulestimate} and the hypothesis of \cref{Lemma:Anders} we have
\begin{align*}
	\langle \Vector{\zeta} , B^{T} \Log|\CC| \rangle \ > \ \log(m) \norm{B \Vector{\zeta}}_1 \ > \ \sum_{i=1}^m \langle \Vector{b(i)}, \Vector{\zeta} \rangle \log \abs{\langle \Vector{b(i)} , \Vector{\zeta} \rangle}
\end{align*}
for all $ \Vector{\zeta} \in (B^{T}M(T))^{\circ} $. (Note that $(B^{T}\NormalCone{\Vector{\eta}})^{\circ} \subset (B^{T}M(T))^{\circ}$.) 
By \cref{Lemma:convex},  we know that the supporting hyperplanes of  $K_{\Vector{\eta}}$ are of the form 
\[ H_{\Vector{\zeta}} \ := \ \left\{  \x \in \R^{m-2n} : \langle \Vector{\zeta} , \x \rangle = \sum_{i}^m  \langle \Vector{b(i)}, \Vector{\zeta}  \rangle \log \abs{ \langle \Vector{b(i)} , \Vector{\zeta} \rangle}  \right\} \]
for some $\Vector{\zeta} \in (B^{T}\NormalCone{\Vector{\eta}})^{\circ} $.  So, these two facts together imply that  $B^{T}\Log|\CC|$ and the shifted copy of $\NormalCone{\Vector{\eta}}$ are not separated by any supporting hyperplane of  $B^{T}K_{\Vector{\eta}}$.  This means $B^{T}\Log|\CC| \in B^{T}K_{\Vector{\eta}}$. Since the kernel of $B^{T}$ is included in $K_{\Vector{\eta}}$ this also implies $\Log|\CC| \in K_{\Vector{\eta}}$.
\end{proof}
\subsection{Putting things together} \label{finalshowdown}
Now we complete the proof of \cref{Proposition:MainCertificateOurAlgorithm}. Recall that by definition
$ \langle \tau , B^{T} \Log \CC \rangle \ = \ \langle B \tau, \Log \CC \rangle$,  and $B(B^{T} M(T))^{\circ} \subseteq M(T)^{\circ}$. So, if a given vector $\Log|\CC|$ satisfies
\begin{align} \label{patch-test}
 \langle \Vector{\zeta} , \Log \CC  \rangle > \log(m) \norm{\Vector{\zeta}}_1  
\end{align}
for all $\Vector{\zeta} \in M(T)^{\circ}$, then by \cref{Lemma:Anders} we have that  $\Log|\CC| \in K_{\Vector{\eta}}$ for a vertex $\Vector{\eta}$ of $\Delta_\A$ which satisfies  $M(T) \subseteq \NormalCone{\Vector{\eta}}$. 

Suppose $M(T)^{\circ}$ is generated by $\Vector{\zeta(1)},\ldots,\Vector{\zeta(L)}$, and assume for a given vector $\Log|\CC|$ we have
\[ \langle \Vector{\zeta(i)} , \Log \CC \rangle > \log(m) \norm{\Vector{\zeta(i)}}_1  \]
for all $i=1,2,\ldots,L$. Then for any  $\x \in M(T)^{\circ}$ with $\x=\sum t_i \Vector{\zeta(i)}$ with $t_i \geq 0$ one has the following inequality
\[ \langle  \Log \CC , \x \rangle \  > \ \log(m) \sum  t_i \norm{\Vector{\zeta(i)}}_1 \geq \log(m) \norm{\x}_1 \]
where the last inequality follows from the triangle inequality. Hence, checking the condition in \cref{Proposition:MainCertificateOurAlgorithm}  only for the generators of $M(T)^{\circ}$ suffices to guarantee 
$ \langle  \Log|\CC| , \x \rangle \  > \log(m) \norm{\x}_1$ for all $x \in M(T)^{\circ}$. At this point we proved that the following: If the hyptohesis of \cref{Proposition:MainCertificateOurAlgorithm} is satisfied, then \cref{Lemma:Anders}  and  \cref{Lemma:convex} show that for all $\Vector{v} \in M(T)$ the ray $\lambda \Vector{v} +  \Log \abs{C}$ for $\lambda \in [0. \infty)$ does not intersect the amoeba of $\Delta_{\A}$. 

Now let $\Gamma$ be a face of $\A$ that is not an irrelevant face. Let $T|_{\Gamma}$ be  the restriction of the triangulation of $T$ on $\Gamma$. By \cref{Lemma:Mixed-cell-cone} the circuit inequalities generating  $M(T|_{\Gamma})^{\circ}$ are included in $M(T)^{\circ}$.  We also observe that $\log(\# \A) \geq \log(\# \Gamma)$. This implies that the criterion of \cref{Proposition:MainCertificateOurAlgorithm} also ensures the ray  $\Vector{v} + \lambda \Log \abs{C}$ does not intersect the amoeba of $\Delta_{\Gamma}$. Using \cref{toriclimit} completes the proof.
\section{Real Polyhedral Homotopy}
\begin{algorithm}[t]
\caption{Real Polyhedral Homotopy}\label{alg:rph}
	\begin{algorithmic}[1]
		\STATE {\bfseries Input} $A_1,A_2,\ldots,A_n \subseteq \mathbb{Z}^n$, $\CC_i \in \R^{A_i}$ for $i=1,2,\ldots,n$ \\
		$\Vector{p}=(p_1,p_2,\ldots,p_n)$ and $p_i=\sum_{\alpha \in A_i} \CC_{i,\alpha} x^{\alpha}$.
		\STATE {\bfseries Initialize} $\A=A_1*A_2*\ldots*A_n$, $\CC=(\CC_1,\CC_2,\ldots,\CC_n)$
		\STATE Set $T$ to be the triangulation of $\A$ that is induced by the lifting function $\Log \CC$\\
		Compute mixed-cells of $T$
		\STATE List generators of $M(T)^{\circ}$
		\IF{For all $\tau \in M(T)^{\circ}$, 
		$$ \langle \Log \CC , \tau \rangle > \log(\# \A) \norm{\tau}_1 $$}
		\STATE Compute real zeros of binomial systems given by mixed-cells of $T$
		\STATE Pick a vector $\Vector{v} \in M(T) - \partial M(T)$ or set $\Vector{v}= \Log \CC$
		\STATE Track solution paths	$x(t)$ for $\Vector{v}$ as in \cref{toriclimit} from $t=0$ to $t=1$ 
		\STATE {\bfseries Output}  $\varietyrealnonzero{\Vector{p}}$
		\ELSE
		\STATE Print ``Input system is not certifiably patchworked"
		\ENDIF
	\end{algorithmic}
\end{algorithm}
In this section we summarize the main steps of our real polyhedral homotopy algorithm. The algorithm follows the common thread of homotopy continuation algorithms, but it operates entirely over the real numbers.

The idea of the algorithm is as follows: Given a polynomial system $\Vector{p}=(p_1,p_2,\ldots,p_n)$ with support sets $A_1,A_2,\ldots,A_n \subseteq \mathbb{Z}^n$ and coefficient vectors $\CC_i \in \R^{\# A_i}$ for $i=1,2,\ldots,n$, we concatanete the support set and coefficient vectors as $\A=A_1*A_2*\ldots*A_n$, $\CC=(\CC_1,\CC_2,\ldots,\CC_n)$. Then, we compute triangulation $T$ of $\A$ with respect to lifting function $\Log \CC$. This step is performed using Jensen's tropical homotopy algorithm as explained \cref{subsection:anders}. Using Jensen's algorithm makes the generators of the cone $M(T)^{\circ}$ readily available. Then we check if the criterion of \cref{Proposition:MainCertificateOurAlgorithm} is satisfied by the vector $\Log \CC$. 
If the criterion is not satisfied, algorithms halts and prints ``Input system is not certifiably patchworked". If the criterion is satisfied, we then find real zeros of binomial systems that correspond to mixed cells of $T$ as explained in \cref{subsection:binomials}. After that we pick a vector $\Vector{v} \in M(T) - \partial M(T)$, this can be done in a multitude of ways e.g. availing to multiplicative updates method, or one can simply set $\Vector{v}=\Log \CC$ since the fact  $\Log \CC \in  M(T) - \partial M(T)$ is already certified. Then we track the solution paths $x(t)$  corresponding to $\Vector{v}$ as in \cref{toriclimit} from $t=0$ to $t=1$. This numerical tracking step is discussed in \cref{subsection:pathtrackers}. The correctness of the algorithm follows from \cref{Proposition:MainCertificateOurAlgorithm} and \cref{toriclimit}.  We give an example showing how the algorithm performs in practice. 
\vspace{-0.1 in}
\begin{example}
\label{Exa:OurAlgorithmInPractice}
	We reconsider the polynomials presented in \cref{Exa:ViroPatchworkingCompleteIntersection}, but this time we fix the coefficients to be real numbers instead of using coefficients that are Puiseux series.
	\begin{align*}
		& f \ = \ x_2^3 - (0.45)x_1x_2^2 - (0.45)^5x_1^2x_2 + (0.45)^{12}x_1^3 - (0.45)x_2^2 + (0.45)^4 x_1 x_2 - (0.45)^9 x_1^2 \\
		&- (0.45)^5 x_2 - (0.45)^9 x_1 + (0.45)^{12},  \\
		& g \ = \ (0.45)^8 x_2^2 - (0.45)^6 x_1 x_2 + (0.45)^6 x_1^2 - (0.45)^3 x_2 - (0.45)^2 x_1 + 1.
	\end{align*}
	This leads to the following support and, using log-absolute values of the coefficients, the following lifting vectors:
	\begin{align*}
		& \text{\texttt{Support f:}} \ \text{\texttt{2$\times$10 Array}$\{$\texttt{Int64,2}$\}$}:
		\left[\begin{array}{cccccccccc}
			0  & 1 & 2 & 3 & 0 & 1 & 2 & 0 & 1 & 0 \\
			3 & 2 & 1 & 0 & 2 & 1 & 0 & 1 & 0 & 0 \\
		\end{array}\right] \\
		& \text{\texttt{Lifting f:}} \
		\left[\begin{array}{cccccccccc}
			0 & 1 & 5 & 12 & 1 & 4 & 9 & 5 & 9 & 12\\
		\end{array}\right] \\
		& \text{\texttt{Support g:}} \ \text{\texttt{2$\times$6 Array}$\{$\texttt{Int64,2}$\}$}:
		\left[\begin{array}{cccccc}
			0 & 1 & 2 & 0 & 1 & 0 \\
			2 & 1 & 0 & 1 & 0 & 0 \\
		\end{array}\right] \\
		& \text{\texttt{Lifting g:}} \
		\left[\begin{array}{cccccc}
			8 & 6 & 6 & 3 & 2 & 0 \\
		\end{array}\right].	
	\end{align*}	
What we did so far corresponds to initialization step of the algorithm (step 2). Now we need to compute mixed cells, and list generators of the dual mixed-cell cone (step 3 and step 4). There are six mixed cells and corresponding circuit inequalities. These mixed cells are depicted in the right picture of \cref{Figure:ViroPatchworking}. After verifying our system is patchworked (this is step 5 in the algorithm), we pass to step 6:  For every one of these mixed cells, we obtain a binomial system, which we then solve using Hermite normal form, e.g., the first mixed-cell is represented by 
	\begin{align*}
	\text{\texttt{volume:} } 1 \qquad \text{\texttt{indices:} }  \text{\texttt{Tuple}$\{$\texttt{Int64,Int64}$\}$}[(2, 1), (5, 6)] \qquad \text{\texttt{normal:} }[-2.0, -1.0]
	\end{align*}
	with a solution for the corresponding binomial system given by
	\begin{align*}
		[4.938271604938272, 2.2222222222222223].
	\end{align*}
	Similarly, we obtain five further solutions for the five other binomial systems corresponding to the other mixed cells. 
	{\small\begin{align*}
		& [4.938271604938272, -0.20249999999999999] \quad [4.938271604938272, -0.041006249999999994] \\
		& [24.386526444139612, 10.973936899862824] \quad [24.386526444139612, -1.0]  \\
		& [24.386526444139612, 0.09112500000000004].
	\end{align*}}
For step 7 we simply pick $v= \Log C$. We perform step 8 using the polyhedral homotopy continuation in \textsc{Homotopy.JL}. After a total runtime of roughly 0.0001 seconds\footnote{Carried out on a MacBook Pro, Intel i5-5257U, 2.70GHz, 8GB RAM.}, we arrive to step 9. Here are  the six real solutions for the original system:
\begin{align*}
		& [4.20818, 2.41707] \quad [7.12063, -0.138875] \quad [6.94337, -0.0383256] \\
		& [49.3211, 24.3919] \quad [15.9697, -0.517115] \quad [17.5735, 0.0244792].
\end{align*}
\end{example}
\vspace{-0.15 in}
\section{Remarks on Complexity} \label{section:complexity}
In this section we discuss complexity aspects of the real polyhedral homotopy algorithm. Our goal in this section is to identify key parameters that governs the complexity of the RPH algorithm. Our main finding is that the complexity of the real polyhedral homotopy algorithm is controlled by the number of mixed cells in triangulation of the Cayley configuration $\A=A_1*A_2*\ldots*A_n$ that is introduced by using the coefficients as a lifting function. We also show that the number of mixed cells admits an $O(t^n)$ upper bound where $t$ is the maximal number of terms in $A_i$ for $1 \leq i \leq n$. So if the number of variables $n$ is considered to be fixed, and the number of terms $t$ is a variable, the discrete computations in RPH takes polynomial time. 

In general, the discrete part of the RPH corresponds to computing mixed cells of a polyhedral subdivision that is induced by a fixed lifting; without worrying about the volumes the mixed cells. Hence, any complexity theoretic upper and lower bounds for computing mixed cells (without volumes) applies to discrete computations in our algorithm. For the numerical part; the number of paths tracked by RPH is dramatically smaller than of complex homotopy algorithms. However, as noted in the introduction we are not able to provide a rigorous complexity analysis for the numerical part of the algorithm for the time being.

\subsection{Tropical homotopy algorithm}
We start this section with bounding the number of inequalities needed to describe a mixed cell.
\begin{lemma} \label{J1}
Let $A_1,A_2,\ldots,A_n$ be point configurations with at most $t$ elements, and let $T$ be a triangulation of $\A=A_1 * A_2 * \ldots * A_n$. Then a 
mixed cell $\sigma \in T$ is determined in the mixed-cell cone $M(\sigma)$ by at most $n(t-2)$ inequalities.
\end{lemma}
\begin{proof}[Proof sketch]
The mixed cell cone describes the case where the simplex corresponding to the mixed cell is a facet of the lifted Cayley polytope. 
So, for every element $\Vector{\alpha} \in \A$ we get a circuit inequality given by $2n$ many vertices of the mixed cell and $\Vector{\alpha}$ that determines whether $\Vector{\alpha}$ is contained in the mixed cell.
In total we have at most $n(t-2)$ many such $\Vector{\alpha}$, and at most that many corresponding circuit inequalities.
\end{proof}

This immediately yields the following corollary.

\begin{corollary} \label{J2}
Let $A_1,A_2,\ldots,A_n$ be point configurations with at most $t$ elements, and let $T$ be a triangulation of $\A= A_1 * A_2 * \ldots * A_n$ with $k$ mixed-cells.
Then the mixed-cell cone $M(T)$ can be described by at most $kn(t-2)$ many linear inequalities all supported on circuits.
\end{corollary}

The proof of \cref{patch_few} gives us an upper bound the number of mixed-cells. Using this rough upper bound we derive the following corollary. 
\begin{corollary} \label{roughcount}
Let $A_1,A_2,\ldots,A_n$ be point configurations with at most $t$ elements, and let $T$ be a triangulation of $\A=A_1 * A_2 * \ldots * A_n$.
Then the mixed-cell cone $M(T)$ can be described by at most $2ne^n(t-1)^{n+1}$ many linear inequalities all supported on circuits. 
\end{corollary}

\cref{roughcount} gives an upper bound to the number of updates in the tropical homotopy algorithm: For a fixed number of variables $n$, it is polynomial in $t$.  This shows that the complexity of a mixed-cell cone computation is controlled by the cardinality of the support sets; this aligns well with Kushnirenko's fewnomial philosophy.

Jensen wrote a paper on implementation details of his algorithm for the purpose of mixed volume computation \cite{jensentropical2}. 
Thanks to real geometry, we do not need volumes, but only the mixed cells. So Jensen's current implementation does not output precisely what we need in this paper. A new implementation that outputs our needs in this paper is currently worked on by Timme. Real polyhedral homotopy is planned to be incorporated into \textsc{Homotopy.JL} \cite{HomotopyJL}. 

\subsection{Effective Viro's patchworking}
As explained in Jensen's paper \cite{jensentropical} and \cite[Lemma 5.1.13]{triangulations}, every circuit inequality is written by a vector with $n+2$ non-zero entries and every entry is given by the volume of a simplex. Since we can compute the volume of a simplex in $O(n^3)$ cost, we can compute each generator of a circuit inequality by $O(n^4)$ cost. This gives us the following basic complexity estimate as a corollary of \cref{J1} and \cref{J2}.
\begin{corollary}
Let $A_1,A_2,\ldots,A_n$ be point configurations with at most $t$ elements, and let $T$ be a triangulation of $\A= A_1 * A_2 * \ldots * A_n$ with $k$ mixed-cells. Then the criterion in \cref{Lemma:Anders} can be checked by $O(kn^5(t-2))$ many arithmetic operations.
\end{corollary}
Using \cref{patch_few} one can provide upper bound for $k$ and hence deduce a $O(e^{n}n^5t^{n+1})$ upper bound for the number of arithmetic operations. 

\subsection{A fewnomial bound for patchworked polynomial systems} \label{section:fewnomial}
We start this section by stating a special case of McMullen's Upper Bound Theorem \cite{ziegler}.
\begin{theorem}[Upper Bound Theorem; special case] \label{upperbound}
Let $Q \subset \R^{2n}$ be a polytope with $t$ vertices. Then the number of facets of $Q$ is bounded by $2\binom{t-n}{n}$.
\end{theorem}

In the case of zero dimensional systems, Viro's method counts the number of common zeros in $(\R^{*})^n$. 
The discussions in \cref{subsection:patchworking} and in \cref{subsection:mixedcell}  show that for a patchworked polynomial system supported with point sets  $A_1,A_2,\ldots,A_n \subset \Z^n$, the number zeros in the positive orthant  is bounded by the number of mixed cells in the corresponding coherent polyhedral subdivision of $A_1+A_2+\ldots+A_{n}$. This yields the following statement. 

\begin{proposition}[Few Zeros for Patchworked Systems] \label{patch_few}
Let $A_1,A_2,\ldots,A_n \subset \Z^n$,  and let $\abs{A_1*A_2* \cdots * A_n} \leq tn$. 
Then for a patchworked polynomial system $\Vector{p}=(p_1,p_2,\ldots,p_n)$ supported with $A_1,A_2,\ldots,A_n$,  the number of common zeros of $\Vector{p}$ in $(\R^{*})^{n}$ is at most 
\begin{align*} 2^{n+1} \binom{tn-n}{n} .
\end{align*} 
\end{proposition}
\begin{proof}
Let $\omega$ be a lifting function and let $\Delta_{\omega}$ be the corresponding coherent fine mixed subdivision of $A_1+A_2+\ldots+A_n$.  The number of mixed cells  in $\Delta_{\omega}$ is equivalent to the number of corresponding simplices in the triangulation of the Cayley configuration $\A=A_1*A_2* \cdots *A_n$; see \cref{subsection:polyhedralcayley}. The simplices that correspond to mixed cells are the simplices with two vertices from each $A_i$. The number of all simplices in the triangulation is less than the number of facets in the lifted Cayley polytope $\Cayley(A^\omega) = \conv(\A^{\omega})$. $\Cayley(A^\omega)$ is contained in $\R^{2n}$, and it has the same number of vertices as $\A$.
So, the number of facets of $\Cayley(A^\omega)$ is bounded by \cref{upperbound}. 
We multiply this bound with $2^n$ to cover all orthants of $(\R^{*})^{n}$, and obtain the following upper bound
\begin{align*} 
	2^{n+1}\binom{tn-n}{n} \ \leq \ 2^{n+1} e^{n} (t-1)^n,  
\end{align*}
where the last inequality follows from Stirling's estimate.
\end{proof}

\subsection{Complexity of numerical path tracking} \label{section:numericalcomplexity}
Homotopy continuation theory of polynomials uses condition numbers to give bounds for the complexity of numerical iterative solvers \cite{condition}. Malajovich noticed that the current theory, which considers solutions of homogeneous polynomials over the projective space, fails to address subtleties of
sparse polynomial systems. He developed a theory of sparse Newton iterations \cite{m2}: For a given  sparse polynomial system $f$, Malajovich's theory uses two condition numbers $\mu(f,\x)$ and $v(\x)$ at a given point $\x \in (\mathbb{C}^{*})^n$, and it provides tools to analyze the accuracy and complexity of sparse Newton iterations. Let us state the main result of Malajovich below.

\begin{theorem}[Malajovich, \cite{m2}] \label{malajovich}
As in \cref{toriclimit}, let $\p_{\CC}(t,\x)$ be the polynomial system. Assume that we track a solution path from $\p_{\CC}(\varepsilon,\x)$ to $\p(\CC)(1,\x)$ where $\varepsilon >0$ is a sufficiently small real number. Then, there exists an algorithm which takes 
\begin{align*} \int_{\varepsilon}^{1} \mu ( \p_{\CC}(t,\x) , \Vector{z}_s ) \; v(\Vector{z}_s) \;  \left( \norm{ \dot{\p}_{\CC,s} }_{\p_{\CC,s}}^{2} + \norm{\dot{\Vector{z}}_s}_{\Vector{z}_s}^2 \right)^{\frac{1}{2}}      \; ds  \end{align*}
many iteration steps where $\Vector{z}_s$ represents the solution path, and $\norm{\cdot}_{\x}$ represents the local norms 
defined as pull-back of the classical Fubini-Study metric under the Veronese map. 
\end{theorem}
It is customary in the theory of homotopy continuation to go from an integral representation as above to a more comprehensible complexity estimate by considering average or smoothed analysis of the iteration process. This amounts to introduce a probability measure on $\p_{\CC}$, the input space of polynomials, and to compute the expectation of the integral estimate over the input space. Malajovich notes in his paper \cite{m2} that the non-existence of a unitary group action on the space of sparse polynomials makes the probabilistic analysis harder. In our opinion, $\mu(\p_{\CC}(t,\x))$ can be analyzed for general measures without group invariance \cite{alp1,alp2}. 
However, the second condition number $v(\x)$ seems hard to analyze; therefore we refrain from a probabilistic analysis for the moment. 

\begin{remark}
Gregorio Malajovich authored a 90 pages Arxiv paper that improves the state of the art \cite{malajovich2020}. Our hope is that these new results will pave the way for a rigorous complexity analysis of RPH, but until then what is written here represents our views.
\end{remark}

\section{Discussion and Outlook}

We discuss some open questions related to this work that are brought to our attention after the initial submission of the article on ArXiv:
\begin{enumerate}
 \item How successful is RPH on practical problems? For instance, how would it perform in problems concerning real polynomial systems coming from chemical reaction networks?
 \item Imagine the support sets $A_1,A_2,\ldots,A_n$ are fixed, and we use i.i.d Gaussian coefficients with unit variance to create a random polynomial system. Can one prove that with high probability 
 the random Gaussian polynomial system would pass your effective patchworking test?
 \item Is our algorithm better than the (complex) polyhedral homotopy algorithm?
 \item How is the comparison of real polyhedral homotopy with Khovanskii-Rolle continuation algorithm of Bates and Sottile?
\end{enumerate}

Regarding the first question: So far, we have only done a preliminary implementation and computed a few examples. The purpose of this article is to provide theoretical foundations for real polyhedral homotopy. We believe that a rigorous implementation and practical testing of the algorithm is crucial, but, given the magnitude of the task, it requires a second, separate article. 

Regarding the second question: In a special case there are explicit estimates that shows indeed with high probability a random polynomial system is a patchworked system \cite{maurice-extremal}. In general, this is a very intriguing question with far reaching consequences: A high probability positive answer would show that Viro's patchworking method captures an essential combinatorial structure that force  patterns on randomly generated systems of polynomial equations.

Regarding the third question: This question was brought to our attention by some colleagues, but the comparison between our algorithm and  the polyhedral homotopy algorithm does not seem to be meaningful. The goal of the two algorithms are different; RPH tracks only real zero paths, and polyhedral homotopy tracks all complex zeros. If one is interested in real roots only, then the advantage of our algorithm is to track correct number of real zero paths (sometimes called optimal path tracking), where most algorithms in the literature find all complex zeros and then filter the real ones. 

For the last question we first need to explain a notable algorithm of Bates and Sottile called \textit{\struc{Khovanskii-Rolle Continuation Algorithm (KR)}} \cite{bates-sottile}. KR admits a sparse polynomial system where every polynomial has at most $t$ terms, and traces at most
\begin{align*} 
	\frac{e^4+3}{4} 2^{\binom{(t-2)n}{2}} \binom{(t-2)n}{t-2,t-2,\ldots,t-2} \sim \exp \left(  t^2 n^2 \right) 
\end{align*}
many solution curves that can lead to real solutions \cite{bates-sottile,bihan-sottile-mixed}.  The number of paths are given by the best fewnomial bound in the literature, and to the best of our knowledge for mixed support the best bounds are in  \cite{bihan-sottile-mixed}.

On the one hand, RPH algorithm tracks polynomially many solution paths with respect to $t$, whereas the KR algorithm traces exponentially many solution curves. For instance, if one needs to solve a system of two bivariate polynomials both with $8$ different terms, the KR algorithm traces more than $2^{76}$ many curves, and RPH tracks less than $2^{12}$ many paths. On the other hand, we stress that the KR algorithm can solve \textit{all} input instances where RPH can only solve polynomials that are located against the discriminant variety. So, in our view these two algorithms are complementary to each other: for a given sparse systems one should use KR when RPH fails to admit the input.

\bibliographystyle{amsalpha}
\bibliography{RealHomotopyContinuation}

\end{document}